\documentclass[leqno,11pt]{amsart}
\usepackage{amssymb, amsmath}
\usepackage{amsthm, amsfonts,mathrsfs}
\usepackage{color}
\def\suite#1#2#3{(#1_{#2})_{#2\in {#3}}}

\def\inte#1{
\displaystyle\mathop{#1\kern0pt}^\circ }

\let\pa=\partial

\let\f=\frac

\let\p=\psi

\let\D=\Delta

\let\wt=\widetilde

\def\cS{{\mathcal S}}

\def\pa{\partial}
\def\grad{\nabla}

\def\virgp{\raise 2pt\hbox{,}}
\def\cdotpv{\raise 2pt\hbox{;}}

\def\eqdefa{\buildrel\hbox{\footnotesize def}\over =}

\def\C{\mathop{\mathbb C\kern 0pt}\nolimits}
\def\DD{\mathop{\mathbb D\kern 0pt}\nolimits}
\def\EE{\mathop{{\mathbb E \kern 0pt}}\nolimits}
\def\K{\mathop{\mathbb K\kern 0pt}\nolimits}
\def\N{\mathop{\mathbb N\kern 0pt}\nolimits}
\def\Q{\mathop{\mathbb Q\kern 0pt}\nolimits}
\def\R{\mathop{\mathbb R\kern 0pt}\nolimits}
\def\SS{\mathop{\mathbb S\kern 0pt}\nolimits}
\def\ZZ{\mathop{\mathbb Z\kern 0pt}\nolimits}
\def\TT{\mathop{\mathbb T\kern 0pt}\nolimits}
\def\P{\mathop{\mathbb P\kern 0pt}\nolimits}

\newcommand{\Z}{{\ZZ}}

\def\dv{\mbox{div}}

\def\dive{\mathop{\rm div}\nolimits}

\def\Supp{\mathop{\rm Supp}\nolimits\ }

\def\no{\noindent}
\def\na{\nabla}
\def\p{\partial}

\newcommand{\beq}{\begin{equation}}
\newcommand{\eeq}{\end{equation}}
\newcommand{\ben}{\begin{eqnarray}}
\newcommand{\een}{\end{eqnarray}}
\newcommand{\beno}{\begin{eqnarray*}}
\newcommand{\eeno}{\end{eqnarray*}}
\newcommand{\andf}{\quad\hbox{and}\quad}
\newcommand{\with}{\quad\hbox{with}\quad}
\newtheorem{defi}{Definition}[section]
\newtheorem{thm}{Theorem}[section]
\newtheorem{lem}{Lemma}[section]
\newtheorem{rmk}{Remark}[section]
\newtheorem{col}{Corollary}[section]
\newtheorem{prop}{Proposition}[section]
\renewcommand{\theequation}{\thesection.\arabic{equation}}
\begin{document}
\title[Fujita-Kato solution of $3-$D inhomogeneous NS equations]
{Global refined Fujita-Kato solution of $3-$D inhomogeneous incompressible Navier-Stokes equations with
large density}

\bigbreak\medbreak
\author[H.  Abidi]{Hammadi Abidi}
\address[H.  Abidi]{D\'epartement de Math\'ematiques
Facult\'e des Sciences de Tunis
Universit\'e de Tunis EI Manar
2092
Tunis
Tunisia}\email{hammadi.abidi@fst.utm.tn}
\author[G. Gui]{Guilong Gui}
\address[G. Gui]{School of Mathematics and Computational Science, Xiangtan University,  Xiangtan 411105,  China}\email{glgui@amss.ac.cn}
\author[P. Zhang]{Ping Zhang}
\address[P. Zhang]{Academy of Mathematics $\&$ Systems Science, The Chinese Academy of Sciences, Beijing 100190, China;\\
School of Mathematical Sciences, University of Chinese Academy of Sciences,
Beijing 100049, China} \email{zp@amss.ac.cn}

\setcounter{equation}{0}
\date{}

\maketitle
\begin{abstract}
We investigate the global unique Fujita-Kato solution to the 3-D inhomogeneous incompressible Navier-Stokes equations, \eqref{1.2}, with  initial velocity $u_0$ being sufficiently small in $\dot{B}^{\frac{1}{2}}_{2, \infty}$ and  with initial density  being bounded from above and below. We first prove the global existence of Fujita-Kato solution to the system \eqref{1.2} if we assume in addition that the initial velocity $u_0\in \dot H^{\f12}.$ While under the additional assumptions that the initial velocity $u_0\in \dot B^{\f12}_{2,1}$  and initial density $\rho_0$ satisfying $\rho_0^{-1}-1\in \dot B^{\f12}_{6,1},$  we prove that $\|\rho^{-1}-1\|_{\widetilde{L}^{\infty}(\R^+;\dot{B}^{\frac{1}{2}}_{6, 1})}$ and $\|u\|_{\widetilde{L}^{\infty}(\R^+;\dot B^{\frac{1}{2}}_{2, 1})\cap {L}^1(\R^+;\dot{B}^{\frac{5}{2}}_{2, 1})}$ are controlled by the norm of the initial data. Our results not only improve
the smallness condition  in the previous references for the initial velocity concerning  the global  Fujita-Kato solution of the system
\eqref{1.2} but also improve the exponential-in-time growth estimate for the solution in \cite{A-G-Z-2} to be the  uniform-in-time estimate \eqref{energy-ineq-total-33}.
\end{abstract}


\vskip 0.2cm


\renewcommand{\theequation}{\thesection.\arabic{equation}}
\setcounter{equation}{0}

\renewcommand{\theequation}{\thesection.\arabic{equation}}
\setcounter{equation}{0}

\section{Introduction}
In this paper, we investigate the global unique Fujita-Kato solution  of the following
$3$-D inhomogeneous incompressible Navier-Stokes equations:
\begin{equation}\label{1.2}
\begin{cases}
\pa_t \rho + \dv (\rho u)=0 \qquad (\forall\,\,(t,x)\in\R^+\times\R^3), \\
\rho(\pa_t u +u\cdot\nabla u )-\Delta u+\grad\Pi=0, \\
\dv\, u = 0, \\
(\rho,u)|_{t=0}=(\rho_0,u_0),
\end{cases}
\end{equation}
where the unknowns $\rho$ and $u=(u_1,u_2, u_3)^{T}$ stand for the density and velocity of the fluid respectively, and $\Pi$  is a  scalar pressure function, which guarantees the divergence free condition of the velocity field. Such a system can be used to describe the mixture of several immiscible fluids that are incompressible and with different densities.
One may check \cite{LP} for more background of this system.

Just as the classical Navier-Stokes equations (NS), which corresponds to the case when $\rho=1$ in \eqref{1.2}, the system \eqref{1.2}  has
the following scaling-invariant property: if $(\rho, u)$ solves \eqref{1.2} with initial data $(\rho_0, u_0)$, then for any $\ell>0$,
\begin{equation}\label{S1eq1}
(\rho, u)_{\ell}(t, x) \eqdefa (\rho(\ell^2\cdot, \ell\cdot), \ell
u(\ell^2 \cdot, \ell\cdot))
\end{equation}
is also a solution of \eqref{1.2} with initial data $(\rho_0(\ell\cdot),\ell
u_0(\ell\cdot))$. We call such functional spaces  as critical spaces if the norms of
which are invariant under the scaling transformation \eqref{S1eq1}.

In \cite{FujitaKato1964},
 Fujita and Kato proved  the following celebrated result for the classical Navier-Stokes equations: given solenoidal vector field $u_0\in \dot{H}^{\frac{1}{2}}(\mathbb{R}^3)$ with  $\|u_0\|_{\dot{H}^{\frac{1}{2}}} \leq \varepsilon_0$ for $ \varepsilon_0>0$ being sufficiently small, (NS) has a unique global solution $u \in C([0, +\infty); \dot{H}^{\frac{1}{2}})\cap L^4(\mathbb{R}^+; \dot{H}^1)\cap L^2(\mathbb{R}^+; \dot{H}^{\frac{3}{2}})$. Cannone, Meyer, and Planchon \cite{CMP1993} extended the result in \cite{FujitaKato1964} for initial data belonging to the general critical Besov space, namely, $u_0\in
 \dot{B}^{-1+\frac{3}{p}}_{p, \infty}.$

Danchin \cite{danchin04} first established the global well-posedness of the system \eqref{1.2} with initial data in the almost critical Sobolev spaces. After the works \cite{A,A-P,DAN-03} in the critical framework, Danchin and Mucha \cite{DM1} eventually proved the global well-posedness of the system \eqref{1.2} with initial density being close enough to a positive constant in the multiplier space of $\dot{B}^{-1+\f{d}p}_{p,1}(\R^d)$ and initial velocity being small enough in $\dot{B}^{-1+\f{d}p}_{p,1}(\R^d)$ for $1\leq p<2d.$ The work of \cite{A-G-Z-2} is the first to investigate the global well-posedness of the system \eqref{1.2} with initial data in the critical spaces and yet without   the size restriction on the initial density. Huang, Paicu and the third author \cite{HPZ-2013} studied the global existence of weak solutions to the system \eqref{1.2} with small initial data $\rho_0-1\in L^{\infty}(\R^d)$ and $u_0 \in \dot{B}^{-1+\frac{d}{p}}_{p, r}(\R^d)$ with $p\in (1,d), \, r \in (1, +\infty)$, and the uniqueness of such solution was only proved under additional regularity assumptions on the initial velocity or on the initial density.  The third author of this paper \cite{Zhang2020} proved the global existence of weak solutions to the system \eqref{1.2} with initial density being bounded from above and below,
and with initial velocity being sufficiently small in the critical Besov space $\dot{B}^{\frac{1}{2}}_{2, 1}(\mathbb{R}^3)$. This solution corresponds to the Fujita-Kato solution of
the classical Navier-Stokes equations. The uniqueness of such solution was proved lately by Danchin and Wang \cite{DW2023}. One may check \cite{DW2023} and references therein for the  progresses in this direction.

Before proceeding, we
  recall the following result concerning the global existence and uniqueness of  Fujita-Kato solution to the system \eqref{1.2} from \cite{DW2023, Zhang2020}.

\begin{thm}[\cite{DW2023,Zhang2020}]\label{thm-Zhang2020}
{\sl
Let $(\rho_0, u_0)$ satisfy $u_0\in \dot{B}^{\frac{1}{2}}_{2, 1}(\mathbb{R}^3)$ with $\dive\,u_0=0,$ and
\begin{equation}\label{bdd-density-assum-1}
0 <c_0\leq \rho_0(x) \leq C_0<+\infty \quad(\forall\,x \in \R^3).
\end{equation}
Then there exists a constant $\varepsilon_0>0$ depending only on $c_0,\,C_0$ such that if
\begin{equation}\label{S1eq3}
\|u_0\|_{\dot{B}^{\frac{1}{2}}_{2, 1}}\leq \varepsilon_0,
\end{equation}
the system \eqref{1.2} has a  unique global solution $(\rho, \, u,\, \na\Pi)$ with $\rho\in L^{\infty}(\mathbb{R}^+\times \mathbb{R}^3)$ and $u\in C([0, +\infty); \dot{B}^{\frac{1}{2}}_{2, 1})\cap \wt{L}^2(\mathbb{R}^+; \dot{B}^{\frac{3}{2}}_{2, 1})$ which satisfies
\begin{equation}\label{bdd-density-res-1}
c_0\leq \rho(t, x) \leq C_0<+\infty \quad (\forall\,(t, x) \in \mathbb{R}^+\times \mathbb{R}^3),
\end{equation}
 and
\begin{align*}
&\|u\|_{\widetilde{L}^{\infty}(\R^+; \dot B^{\frac{1}{2}}_{2, 1})}
+\|u\|_{\widetilde{L}^2(\R^+; \dot{B}^{\frac{3}{2}}_{2, 1})}
+\|t^{\frac{1}{2}}(\nabla u, \Pi)\|_{{L}^2(\R^+; \dot{B}^{\frac{1}{2}}_{6, 1})}\\
&+\|t^{\frac{1}{2}}
u\|_{\widetilde{L}^{\infty}(\R^+; \dot B^{\frac{3}{2}}_{2, 1})}+\|t^{\frac{1}{2}}u_t\|_{\widetilde{L}^2(\R^+; \dot{B}^{\frac{1}{2}}_{2, 1})}+\|t^{\frac{1}{2}}(\nabla^2 u, \nabla\Pi)\|_{{L}^2(\R^+; L^3)}\\
&+\|tu_t\|_{\widetilde{L}^{\infty}(\R^+; \dot B^{\frac{1}{2}}_{2, 1})}+\|tD_tu\|_{\widetilde{L}^2(\R^+; \dot{B}^{\frac{3}{2}}_{2, 1})}\leq C  \|u_0\|_{\dot{B}^{\frac{1}{2}}_{2,
1}},
\end{align*}
where $D_tu\eqdefa (\partial_t+u\cdot\nabla)u,$ and the definition of Chemin-Lerner space $\wt{L}^2(\mathbb{R}^+; \dot{B}^{\frac{3}{2}}_{2, 1})$ can be found in Definition \ref{def2.2}.
 }
\end{thm}
Very recently, instead of using $\nabla\,u_2 \in L^1(0, T; L^{\infty})$, Hao et al. \cite{HSWZ2024} used $t^{\frac{1}{2}}\nabla\,u_2 \in L^2(0, T; L^{\infty})$ to derive the uniqueness  of solutions to the system \eqref{1.2}. Precisely, they proved the following interesting uniqueness result:

\begin{thm}[Theorem 1.3 in \cite{HSWZ2024}]\label{thm-unique-HSWZ-1}
{\sl  Let $(\rho_1, u_1, \Pi_1)$ and $(\rho_2, u_2, \Pi_2)$ be two solutions of the system \eqref{1.2} on $[0, T ]\times \mathbb{R}^3$
corresponding to the same initial data. Assume in addition that:
\begin{itemize}
\item $\sqrt{\rho_1} (u_2 -u_1)\in L^{\infty}(0, T; L^2(\mathbb{R}^3))$;
\item $ \nabla\,(u_2- u_1)\in L^2((0, T)\times \mathbb{R}^3)$;
\item $t^{\frac{1}{2}}\nabla\,u_2 \in L^2(0, T; L^{\infty})$;
\item $ \nabla u_2 \in L^4(0, T; L^{2})$;
\item $t^{\frac{3}{4}}\,D_t u_2 \in L^2(0, T; L^{2})$.
\end{itemize}
Then, $(\rho_1, u_1, \nabla\Pi_1) \equiv (\rho_2, u_2, \nabla\Pi_2)$ on $[0, T ] \times \mathbb{R}^3$.
}
\end{thm}
Based on Theorem \ref{thm-unique-HSWZ-1} and motivated by  \cite{Zhang2020}, Hao et al. \cite{HSWZ2024} proved the  global unique solution to the system \eqref{1.2} with initial density being bounded from above and below and with  initial velocity being sufficiently small in $\dot{H}^{\frac{1}{2}}$.

 Compared with the result of global Fujita-Kato solution to the classical 3-D Navier-Stokes equations,
 the smallness condition of the initial velocity $u_0$ in $\dot{B}^{\frac{1}{2}}_{2, \infty}$ seems more natural
  for the system \eqref{1.2} in the $L^2$ framework.
 The goal of this paper is to show the global unique Fujita-Kato solution of \eqref{1.2} with the initial data in the critical Besov spaces but without smallness assumption on the initial density $\rho_0$ yet with  initial velocity $u_0$ being small enough in $\dot{B}^{\frac{1}{2}}_{2, \infty}$.

 Our main results state as follows.

 \begin{thm}
 \label{thmmain-global}
{\sl  Let $(\rho_0, u_0)$ satisfy  \eqref{bdd-density-assum-1} and $u_0\in \dot{H}^{\frac{1}{2}}(\mathbb{R}^3)$ with $\dive\,u_0=0.$
Then there exists a constant $\frak{c}>0$ depending only on $c_0,\,C_0$ such that if
\begin{equation}\label{smallness-u-1}
\|u_0\|_{\dot{B}^{\frac{1}{2}}_{2, \infty}}\leq \frak{c},
\end{equation}
the system \eqref{1.2} has a global unique solution $(\rho, \, u,\, \nabla\Pi)$ with $\rho\in C_{\rm w}([0,\infty); L^{\infty})$ and $u\in C([0, +\infty); \dot{H}^{\frac{1}{2}})\cap L^2(\mathbb{R}^+; \dot{H}^{\frac{3}{2}})$ which satisfies \eqref{bdd-density-res-1} and
 \begin{equation}\label{unif-est-1-1}
\begin{split}
& \|u\|_{\widetilde{L}^{\infty}(\mathbb{R}^+; \dot{B}^{\frac{1}{2}}_{2, r})}+\|u\|_{\widetilde{L}^2(\mathbb{R}^+; \dot{B}^{\frac{3}{2}}_{2, r})}\lesssim  \|u_0\|_{\dot{B}^{\frac{1}{2}}_{2, r}} \quad \mbox{for any} \quad r\in[2, +\infty],
\end{split}
\end{equation}
and
 \begin{equation}\label{unif-est-1-2}
\begin{split}
& \|\nabla\,u\|_{L^4(\mathbb{R}^+; L^2)}+\|t^{\frac{1}{4}}\nabla\,u\|_{L^\infty(\mathbb{R}^+; L^{2})}\lesssim \|u_0\|_{\dot{B}^{\frac{1}{2}}_{2, \infty}}, \\
&\|t^{\frac{1}{2}} u\|_{\widetilde{L}^\infty(\mathbb{R}^+; \dot{B}^{\frac{3}{2}}_{2, r})}+\|t^{\frac{1}{2}}\partial_t u\|_{\widetilde{L}^2(\mathbb{R}^+; \dot{B}^{\frac{1}{2}}_{2, r})}\lesssim \|u_0\|_{\dot{B}^{\frac{1}{2}}_{2, r}} \quad \forall\,\ r \in [2, +\infty],\\
&\|{t}^{\frac{1}{4}}\,u_{t}\|_{L^2(\R^+; L^{2})}+\|t^{\frac{3}{4}}(\partial_t u,\na^2u, D_t u)\|_{L^\infty(\mathbb{R}^+; L^2)}+\|t^{\frac{1}{2}}\nabla\,u\|_{L^2(\mathbb{R}^+;  L^{\infty})}\\
&+\|t^{\frac{3}{4}}\nabla
(\partial_t u, D_t u)\|_{L^2(\mathbb{R}^+; L^2)} +\|t^{\frac{3}{4}}(\nabla^2 u, \,\nabla \Pi)\|_{L^2(\mathbb{R}^+; L^6)}\lesssim\|u_0\|_{\dot{H}^{\frac{1}{2}}}.
\end{split}
\end{equation}
}
\end{thm}

\begin{thm}[Regularity for Lipschitz flows]\label{thmmain-regularity}
{\sl  Let $(\rho_0, u_0)$ satisfy  \eqref{bdd-density-assum-1} and
\begin{equation}\label{S1eq2}
 \rho_{0}^{-1}-1 \in \dot{B}^{\frac{1}{2}}_{6, 1}(\mathbb{R}^3),
\, u_0\in \dot{B}^{\frac{1}{2}}_{2, 1}(\mathbb{R}^3) \with \dive\,u_0=0.
\end{equation}
Then there exists a constant $\frak{c}>0$ depending only on $c_0,\,C_0$ such that if \eqref{smallness-u-1} holds, the system \eqref{1.2} has a  unique global solution $(\rho, \, u,\, \nabla\Pi)$ which satisfies \eqref{bdd-density-res-1} and
\begin{equation}\label{energy-ineq-total-22}
(\rho^{-1}-1, u) \in C([0,\infty); \dot{B}^{\frac{1}{2}}_{6, 1}) \times \bigl(C([0,\infty); \dot{B}^{\frac{1}{2}}_{2, 1})
\cap {L}^1(\mathbb{R}^+; \dot{B}^{\frac{5}{2}}_{2, 1})\bigr),
\end{equation}
and
\begin{equation}\label{energy-ineq-total-33}
\begin{split}
&\|\rho^{-1}-1\|_{\widetilde{L}^{\infty}(\R^+;\dot{B}^{\frac{1}{2}}_{6, 1})}+\|
u\|_{\widetilde{L}^{\infty}(\R^+;\dot B^{\frac{1}{2}}_{2, 1})}
+\|u\|_{{L}^1(\R^+;\dot{B}^{\frac{5}{2}}_{2, 1})}
\lesssim
\|\rho_0^{-1}-1\|_{\dot{B}^{\frac{1}{2}}_{6, 1}}
+\|u_0\|_{\dot{B}^{\frac{1}{2}}_{2,
1}}.
\end{split}
\end{equation}
}
\end{thm}

\begin{rmk} \begin{itemize}
\item[(1)] In our previous paper \cite{A-G-Z-2}, we proved that if $(\rho_0, u_0)$ satisfy  \eqref{bdd-density-assum-1}, \eqref{S1eq3},
$
 \rho_{0}^{-1}-1 \in \dot{B}^{\frac{3}{2}}_{2, 1}$ and $\dive\,u_0=0,$
the system \eqref{1.2} has a unique global solution $(\rho,u)$ so that for any $t>0,$
\beq \label{S1eq4}
\begin{split}
\|\rho^{-1}-1\|_{\widetilde{L}^{\infty}_t(\dot{B}^{\frac{3}{2}}_{2, 1})}&+\|
u\|_{\widetilde{L}^{\infty}_t(\dot B^{\frac{1}{2}}_{2, 1})}
+\|u\|_{{L}^1_t(\dot{B}^{\frac{5}{2}}_{2, 1})}\\
&\lesssim
\bigl(\|\rho_0^{-1}-1\|_{\dot{B}^{\frac{3}{2}}_{2, 1}}
+\|u_0\|_{\dot{B}^{\frac{1}{2}}_{2,
1}}\bigr)\exp\bigl(C\sqrt{t}\bigr).
\end{split}
\end{equation}
We emphasize that Theorem \ref{thmmain-regularity} improves not only the smallness condition from \eqref{S1eq3}
to  \eqref{smallness-u-1} but also the exponential-in- time growth estimate \eqref{S1eq4} to the one \eqref{energy-ineq-total-33}
which is uniform with respect to time.

\item[(2)] We remark that the key ingredient in the proof of Theorem \ref{thmmain-regularity}
is to derive the {\it a priori} estimate:\begin{equation}\label{S1eq5}
\|\nabla u\|_{L^1(\R^+;L^\infty)} \lesssim\|u_0\|_{\dot{B}^{\frac{1}{2}}_{2, 1}}.
\end{equation}
The proof of \eqref{S1eq5} is motivated by \cite{DW2023} and yet simpler here. Once again our smallness
condition for the velocity field \eqref{smallness-u-1} improves the one \eqref{S1eq3} in \cite{DW2023}.
\end{itemize}
\end{rmk}

\no{\bf Scheme of the proof and organization of the paper.}

In Section \ref{Sect2}, we present some laws of product and commutator's estimates, which will be used in the subsequent sections.

In Section \ref{Sect3}, we present the proof of Theorem \ref{thmmain-global}.

In Section \ref{Sect4}, we present the proof of Theorem \ref{thmmain-regularity}.

Finally in the Appendix \ref{Sect5}, we collect some basic facts on Littlewood-Paley theory and Lorentz spaces.\\

Let us complete this section with the notations that we are going to use
in this context.

\medbreak \noindent{\bf Notations:} Let $A, B$ be two operators, we
denote $[A,B]=AB-BA,$ the commutator between $A$ and $B$. For
$a\lesssim b$, we mean that there is a uniform constant $C,$ which
may be different on different lines, such that $a\leq Cb$.

For $X$ a Banach space and $I$ an interval of $\R,$ we denote by
$C(I;\,X)$ the set of continuous functions on $I$ with
values in $X.$  For $q\in[1,+\infty],$ the
notation $L^q(I;\,X)$ stands for the set of measurable functions on
$I$ with values in $X,$ such that $t\longmapsto\|f(t)\|_{X}$ belongs
to $L^q(I).$

We always denote   $\{c_{j, r}\}_{j \in \mathbb{Z}} $ a nonnegative generic
element in the unit sphere of $\ell^r(\Z)$ and $\suite c j \ZZ$
(resp.~$\suite d j\ZZ$) a nonnegative generic element in the unit sphere
of~$\ell^2(\ZZ)$ (resp.~$\ell^1(\Z)$).

\renewcommand{\theequation}{\thesection.\arabic{equation}}
\setcounter{equation}{0}

\setcounter{equation}{0}
\section{Preliminaries}\label{Sect2}

The proofs  of Theorems \ref{thmmain-global} and \ref{thmmain-regularity} require Littlewood-Paley
theory, which we recall  in Appendix \ref{Sect5} for the convenience of readers. Below we shall
 apply the basic facts in  Appendix \ref{Sect5} to study some
estimates, which will be used in the subsequent sections. We first present the following law of product, the proof of which is given  for the sake of completeness.

\begin{lem}\label{lem-product-1-1}
{\sl Let $f \in \dot{B}^{\frac{1}{2}}_{2, \infty}(\mathbb{R}^3)$ and  $ g \in \dot{H}^1(\mathbb{R}^3)$, then one has
\begin{equation}\label{est-lem-prod-1}
\|f\,g\|_{L^2} \lesssim \|f\|_{\dot{B}^{\frac{1}{2}}_{2, \infty}}\|g\|_{\dot{H}^1}.
\end{equation}}
\end{lem}
\begin{proof}
By virtue of Bony's decomposition \eqref{bony}, we decompose $f\,g$ into  three terms:
\begin{equation}\label{est-prod-1-0}
\begin{split}
 f\,g=T_f g+T_{g}f+R(f, g).
 \end{split}
\end{equation}
We first observe from Lemma \ref{lem2.1} that
\begin{equation}\label{est-Sk-1}
\|\dot{S}_{k-1}f\|_{L^\infty}\lesssim\sum_{\ell \leq k-2}\|\dot{\Delta}_{\ell}f\|_{L^\infty}\lesssim  \sum_{\ell \leq k-2}2^{\frac{3}{2}\ell}\|\dot{\Delta}_{\ell}f\|_{L^2}\lesssim 2^k\|f\|_{\dot{B}^{\frac{1}{2}}_{2, \infty}},
\end{equation}
from which, we infer
\begin{equation*}
\begin{split}
\|\dot{\Delta}_qT_fg \|_{L^2} &\lesssim \sum_{|q-k|\leq 4}\|\dot{S}_{k-1}f\|_{L^\infty} \|\dot{\Delta}_kg\|_{L^2}\\
&\lesssim  \sum_{|q-k|\leq 4}c_k\|g\|_{\dot{H}^1} \|f\|_{\dot{B}^{\frac{1}{2}}_{2, \infty}}\lesssim c_q  \|f\|_{\dot{B}^{\frac{1}{2}}_{2, \infty}}\|g\|_{\dot{H}^1}.
 \end{split}
\end{equation*}
Henceforth we always denote $\left(c_q\right)_{q\in\Z}$ to be a nonnegative generic element of $\ell^2(\Z)$ so that $\sum_{q\in\Z}c^2_q=1.$

Similarly, due to
$$ \|\dot{S}_{k-1}g\|_{L^\infty}\lesssim c_k2^{\f{k}2}\|g\|_{\dot{H}^1},$$
one has
\begin{equation*}
\begin{split}
\|\dot{\Delta}_qT_{g}f \|_{L^2} \lesssim & \sum_{|q-k|\leq 4}\|\dot{S}_{k-1}g\|_{L^\infty}\|\dot{\Delta}_kf\|_{L^2}
\lesssim  c_{q} \|f\|_{\dot{B}^{\frac{1}{2}}_{2, \infty}}\|g\|_{\dot{H}^1}.
 \end{split}
\end{equation*}

Finally, we deduce from Lemma \ref{lem2.1} that
\begin{equation*}
\begin{split}
\|\dot{\Delta}_q R(f, g)\|_{L^2} \lesssim & 2^{\frac{3}{2}q}\sum_{k\geq q-3}\|\dot{\Delta}_kf\|_{L^2}\|\widetilde{\dot{\Delta}}_kg\|_{L^2}\\
\lesssim &\|g\|_{\dot{H}^1}\|f\|_{\dot{B}^{\frac{1}{2}}_{2, \infty}} 2^{\frac{3}{2}q}\sum_{k\geq q-3} c_{k} 2^{-\frac{3}{2}k}\lesssim  c_{q}\|f\|_{\dot{B}^{\frac{1}{2}}_{2, \infty}}\|g\|_{\dot{H}^1}.
 \end{split}
\end{equation*}

By summarizing the above estimates, we reach \eqref{est-lem-prod-1}.
\end{proof}

\begin{lem}\label{lem-product-uu-2aaa}
{\sl Let   $f \in {\dot{B}^{\frac{1}{2}}_{2, \infty}}
\cap{\dot{B}^{\frac{3}{2}}_{2, \infty}}(\mathbb{R}^3)$, then there holds
\begin{equation}\label{product-uu-2aaa1-1}
\begin{split}
&\|f^2\|_{\dot{B}^{\frac{1}{2}}_{2, \infty}}\lesssim \|f \|_{\dot{B}^{\frac{1}{2}}_{2, \infty}}\|f\|_{\dot{B}^{\frac{3}{2}}_{2, \infty}}.
\end{split}
\end{equation}}
\end{lem}

\begin{proof} In view of \eqref{bony}, we write
\beq \label{S2eq1}
f^2=2T_ff+R(f, f). \eeq
We first get from \eqref{est-Sk-1} that
\begin{equation*}\label{product-uu-2aaa1-2}
\begin{split}
\|\dot{\Delta}_q T_ff\|_{L^2}&\lesssim \sum_{|q-k|\leq 4}\|\dot{S}_{k-1}f\|_{L^\infty}\|\dot{\Delta}_kf\|_{L^2}\\
&\lesssim \|f \|_{\dot{B}^{\frac{1}{2}}_{2, \infty}}\|f\|_{\dot{B}^{\frac{3}{2}}_{2, \infty}}\sum_{|q-k|\leq 4}2^{-\frac{k}{2}}\lesssim 2^{-\frac{q}{2}}\|f\|_{\dot{B}^{\frac{1}{2}}_{2, \infty}}\|f\|_{\dot{B}^{\frac{3}{2}}_{2, \infty}}.
\end{split}
\end{equation*}

On the other hand, we deduce from Lemma \ref{lem2.1} that
\begin{equation*}
\begin{split}
\|\dot{\Delta}_q R(f, f)\|_{L^2} &\lesssim 2^{\frac{3}{2} q}\sum_{ k\geq q-3}\|\dot{\Delta}_kf\|_{L^2}\| \widetilde{\dot{\Delta}}_kf\|_{L^2}\\
& \lesssim \|f \|_{\dot{B}^{\frac{1}{2}}_{2, \infty}}\|f \|_{\dot{B}^{\frac{3}{2}}_{2, \infty}}\, 2^{\frac{3}{2} q}\sum_{ k\geq q-3} 2^{-2k} \lesssim \,2^{-\frac{q}{2}} \|f\|_{\dot{B}^{\frac{1}{2}}_{2, \infty}}\|f\|_{\dot{B}^{\frac{3}{2}}_{2, \infty}}.
\end{split}
\end{equation*}

By summarizing the above estimates, we obtain \eqref{product-uu-2aaa1-1}.
\end{proof}

\begin{lem}\label{lem-product-uu-1}
{\sl Let $p \in [2, +\infty)$, $f\in {L}^{\infty}_t(\dot{B}^{-1+\frac{3}{p}}_{p, \infty}(\mathbb{R}^3))\cap \widetilde{L}^{1}_t(\dot{B}^{1+\frac{3}{p}}_{p, \infty}(\mathbb{R}^3))$, then we have
\begin{equation}\label{est-basic-2-237}
\begin{split}
&\|f^2\|_{\widetilde{L}^1_t(\dot{B}^{\frac{3}{p}}_{p, \infty})}\lesssim \|f\|_{{L}^{\infty}_t(\dot{B}^{-1+\frac{3}{p}}_{p, \infty})}\|f\|_{\widetilde{L}^{1}_t(\dot{B}^{1+\frac{3}{p}}_{p, \infty})}.
\end{split}
\end{equation}}
\end{lem}
\begin{proof} Observing that
\begin{align*}
\|\dot{S}_{k-1}f\|_{L^\infty_t(L^{\infty})}
\lesssim &\bigl\|\sum_{\ell \leq k-2}2^{\frac{3}{p}\ell}\|\dot{\Delta}_{\ell}f(t)\|_{L^{p}}\bigr\|_{L^\infty_t}\lesssim
2^k\|f \|_{{L}^{\infty}_t(\dot{B}^{-1+\frac{3}{p}}_{p, \infty})},
\end{align*}
so that one has
\begin{equation*}
\begin{split}
\|\dot{\Delta}_q T_ff\|_{L^1_t(L^p)}
&\lesssim \sum_{|q-k|\leq 4}\|\dot{S}_{k-1}f\|_{L^\infty_t(L^{\infty})} \|\dot{\Delta}_kf\|_{L^1_t(L^p)}\\
&\lesssim \|f\|_{\widetilde{L}^{1}_t(\dot{B}^{1+\frac{3}{p}}_{p, \infty})}\|f \|_{{L}^{\infty}_t(\dot{B}^{-1+\frac{3}{p}}_{p, \infty})}
\sum_{|q-k|\leq 4}2^{-\frac{3}{p}k}\\
&\lesssim 2^{-\frac{3}{p}q}\|f\|_{\widetilde{L}^{1}_t(\dot{B}^{1+\frac{3}{p}}_{p, \infty})}\|f\|_{{L}^{\infty}_t(\dot{B}^{-1+\frac{3}{p}}_{p, \infty})}.
\end{split}
\end{equation*}

On the other hand, we observe from Lemma \ref{lem2.1} that
\begin{equation*}
\begin{split}
\|\dot{\Delta}_q R(f, f)\|_{L^1_t(L^p)} &\lesssim  2^{\frac{3}{p} q}\sum_{ k\geq q-3}\|\dot{\Delta}_kf\|_{L^{\infty}_t(L^{p})}\| \widetilde{\dot{\Delta}}_kf\|_{L^1_t(L^{p})}\\
& \lesssim \|f\|_{{L}^{\infty}_t(\dot{B}^{-1+\frac{3}{p}}_{p, \infty})}\|f\|_{\widetilde{L}^{1}_t(\dot{B}^{1+\frac{3}{p}}_{p, \infty})} 2^{\frac{3}{p} q}\sum_{ k\geq q-3} 2^{-\frac{6}{p}k}\\
 & \lesssim 2^{-\frac{3}{p} q}\|f\|_{\widetilde{L}^{\infty}_t(\dot{B}^{-1+\frac{3}{p}}_{p, \infty})}\|f\|_{\widetilde{L}^{1}_t(\dot{B}^{1+\frac{3}{p}}_{p, \infty})}.
\end{split}
\end{equation*}

By combining the above estimates with \eqref{S2eq1},
 we achieve \eqref{est-basic-2-237}.
\end{proof}

\begin{lem}\label{lem-product-au-2}
{\sl Let $p \in [2, 3]$ and $\lambda \in (3, +\infty)$ satisfy $\frac{3}{p}-\frac{3}{\lambda}\leq 1$. Let $ f\in {L}^{\infty}_t\bigl(\dot{B}^{\frac{3}{\lambda}}_{\lambda, \infty}\cap L^{\infty}(\mathbb{R}^3)\bigr)$ and $g\in \widetilde{L}^{1}_t\bigl(\dot{B}^{-1+\frac{3}{p}-\frac{3}{\lambda}}_{\frac{\lambda\,p}{\lambda-p}, 1}\cap \dot{B}^{-1+\frac{3}{p}}_{p, \infty}(\mathbb{R}^3)\bigr).$ Then one has
\begin{equation}\label{est-product-au-1}
\begin{split}
& \|f\,g\|_{\widetilde L^1_t(\dot{B}^{-1+ \frac{3}{p} }_{p, \infty})}  \lesssim  \|f\|_{L^{\infty}_t(L^{\infty})}\|g\|_{\widetilde{L}^{1}_t(\dot{B}^{-1+\frac{3}{p}}_{p, \infty})}+\|f\|_{{L}^{\infty}_t(\dot{B}^{\frac{3}{\lambda}}_{\lambda, \infty})}\|g\|_{{L}^{1}_t\bigl(\dot{B}^{-1+\frac{3}{p}-\frac{3}{\lambda}}_{\frac{\lambda\,p}{\lambda-p}, 1}\bigr)}.
\end{split}
\end{equation}}
\end{lem}

\begin{proof}
We first get, by applying Proposition \ref{prop2.2}, that for any $q \in \mathbb{Z}$,
\begin{align*}
&\|\dot{\Delta}_q T_fg\|_{L^1_t(L^p)}\lesssim 2^{(1-\frac{3}{p}) q}\|f\|_{L^{\infty}_t(L^{\infty})}\|g\|_{\widetilde{L}^{1}_t(\dot{B}^{-1+\frac{3}{p}}_{p, \infty})},\\
&\|\dot{\Delta}_q R(f, g)\|_{L^1_t(L^p)} \lesssim 2^{(1-\frac{3}{p}) q} \|f\|_{{L}^{\infty}_t(\dot{B}^{\frac{3}{\lambda}}_{\lambda, \infty})}\|g\|_{{L}^{1}_t\bigl(\dot{B}^{-1+\frac{3}{p}-\frac{3}{\lambda}}_{\frac{\lambda\,p}{\lambda-p}, 1}\bigr)}.
\end{align*}

While due to $\frac{3}{p}-\frac{3}{\lambda}\leq 1,$ we have
\begin{align*}
\|\dot{S}_{k-1}g\|_{L^1_t(L^{\frac{\lambda\,p}{\lambda-p}})}\lesssim&\|g\|_{{L}^{1}_t\bigl(\dot{B}^{-1+\frac{3}{p}-\frac{3}{\lambda}}_{\frac{\lambda\,p}{\lambda-p}, 1}\bigr)} \sum_{\ell \leq k-2}2^{\left(1-\frac{3}{p}+\frac{3}{\lambda}\right)\ell}d_{\ell}\\
\lesssim&2^{\left(1-\frac{3}{p}+\frac{3}{\lambda}\right)k}\|g\|_{{L}^{1}_t\bigl(\dot{B}^{-1+\frac{3}{p}-\frac{3}{\lambda}}_{\frac{\lambda\,p}{\lambda-p}, 1}\bigr)},
\end{align*}
so that one has
\begin{align*}
\|\dot{\Delta}_q T_{g}f\|_{L^1_t(L^p)} &\lesssim  \sum_{|q-k|\leq 4}\|\dot{S}_{k-1}g\|_{L^1_t(L^{\frac{\lambda\,p}{\lambda-p}})}\|\dot{\Delta}_kf\|_{L^\infty_t(L^{\lambda})}\\
 &\lesssim  \|f\|_{{L}^{\infty}_t(\dot{B}^{\frac{3}{\lambda}}_{\lambda, \infty})}\|g\|_{{L}^{1}_t(\dot{B}^{-1+\frac{3}{p}-\frac{3}{\lambda}}_{\frac{\lambda\,p}{\lambda-p}, 1})}\sum_{|q-k|\leq 4}
  2^{\left(1-\frac{3}{p}\right)k} \\
&\lesssim  \|f\|_{{L}^{\infty}_t(\dot{B}^{\frac{3}{\lambda}}_{\lambda, \infty})}\|g\|_{{L}^{1}_t(\dot{B}^{-1+\frac{3}{p}-\frac{3}{\lambda}}_{\frac{\lambda\,p}{\lambda-p}, 1})} 2^{(1-\frac{3}{p}) q}.
\end{align*}
Henceforth we always denote $\left(d_q\right)_{q\in\Z}$ to be a nonnegative generic element of $\ell^1(\Z)$ so that $\sum_{q\in\Z}d_q=1.$

By summarizing the above estimates and using \eqref{est-prod-1-0}, we achieve
 \eqref{est-product-au-1}.
\end{proof}

\begin{lem}\label{lem-cumm-1}
{\sl  Let $2 \leq  p \leq\infty$ and $u\in {L}^{1}_t(\dot{B}^{1}_{\infty, 1}(\mathbb{R}^3))\cap {L}^{\infty}_t(\dot{B}^{\frac{1}{2}}_{2, \infty}(\mathbb{R}^3))$ with $\dive u=0$. Then  there holds
\begin{equation}\label{est-basic-2-121}
\begin{split}
\sum_{q\in\mathbb{Z}}2^{q\left(-1+\frac{3}{p}\right)}\|[\dot\Delta_{q}, u\cdot \nabla ]u\|_{L^1_t(L^p)}
\lesssim \|u\|_{{L}^{1}_t(\dot{B}^{1}_{\infty, 1})} \|u\|_{{L}^{\infty}_t(\dot{B}^{\frac{1}{2}}_{2, \infty})}.
\end{split}
\end{equation}}
\end{lem}

\begin{proof}
By applying Bony's decomposition \eqref{bony} and
the divergence free
condition of $u$, we write
\begin{equation}\label{comma-1}
\begin{split}
&[\dot{\Delta}_q , u\cdot\nabla] u=\dot{\Delta}_q\bigl(\partial_j
R(u^j, u))+ \dot{\Delta}_q\bigl(T_{\partial_j u}u^j\bigr)
-T'_{\dot{\Delta}_q\partial_j u}u^j+[\dot{\Delta}_q, T_{u^j}]\partial_j u\eqdefa
\sum_{i=1}^4 I^{(i)}_q.
\end{split}
\end{equation}
Here repeated upper and lower indices mean summation from $1$ to $3.$ Let us handle term by term above.

We first deduce  from Lemma \ref{lem2.1} that
\begin{equation*}
\begin{split}
\|I^{(1)}_q\|_{L^1_t(L^p)}
&\lesssim 2^{q}2^{q\left(\frac{3}{2}-\frac{3}{p}\right)}\sum_{k\geq q-3} \|\widetilde{\dot{\Delta}}_ku\|_{L^\infty_t(L^2)}\|\dot{\Delta}_ku\|_{L^1_t(L^\infty)}\\
&\lesssim \|u\|_{{L}^{1}_t(\dot{B}^{1}_{\infty, 1})} \|u\|_{{L}^{\infty}_t(\dot{B}^{\frac{1}{2}}_{2, \infty})}2^{q\left(\frac{5}{2}-\frac{3}{p}\right)}\sum_{k\geq q-3} d_k 2^{-\frac{3}{2}k}\\
&\lesssim d_{q} 2^{q\left(1-\frac{3}{p}\right)}\|u\|_{{L}^{1}_t(\dot{B}^{1}_{\infty, 1})} \|u\|_{{L}^{\infty}_t(\dot{B}^{\frac{1}{2}}_{2, \infty})}.
\end{split}
\end{equation*}

Whereas it follows from Lemma \ref{lem2.1} that
\begin{align}\label{S2eq3}
\|\dot{S}_{k-1}
u\|_{L^{\infty}_t(L^p)}\lesssim \bigl\|\sum_{\ell\leq k-2} 2^{\left(\frac{5}{2}-\frac{3}{p}\right)\ell}\|\dot{\Delta}_\ell
u\|_{L^2}\bigr\|_{L^{\infty}_t}\lesssim \|u\|_{{L}^{\infty}_t(\dot{B}^{\frac{1}{2}}_{2, \infty})}2^{\left({2}-\frac{3}{p}\right)k},
\end{align}
from which, we infer
\begin{equation*}
\begin{split}
\|I^{(2)}_q\|_{L^1_t(L^p)}&\lesssim \sum_{\vert
q-k\vert\leq 4}\|\dot{S}_{k-1}
u\|_{L^{\infty}_t(L^p)} \|\dot{\Delta}_ku\|_{L^1_t(L^\infty)}\\
&\lesssim\|u\|_{{L}^{1}_t(\dot{B}^{1}_{\infty, 1})} \|u\|_{{L}^{\infty}_t(\dot{B}^{\frac{1}{2}}_{2, \infty})} \sum_{\vert
q-k\vert\leq 4} d_{k} 2^{\left(1-\frac{3}{p}\right)k}\\
&\lesssim d_{q} 2^{q\left(1-\frac{3}{p}\right)}\|u\|_{{L}^{1}_t(\dot{B}^{1}_{\infty, 1})} \|u\|_{{L}^{\infty}_t(\dot{B}^{\frac{1}{2}}_{2, \infty})}.
\end{split}
\end{equation*}

 Observing that $I^{(3)}_q=-\sum_{k\geq
q-3}\dot{S}_{k+2}\dot\Delta_q\partial_j u\dot\Delta_k u^j$, one has
\begin{equation*}
\begin{split}
\Vert {I}^{(3)}_q\Vert_{L^1_t(L^p)}&\lesssim \|\dot{\Delta}_q\partial_j u\|_{L^1_t(L^p)}\sum_{k\geq
q-3} \|\dot{\Delta}_k u^j\|_{L^1_t(L^\infty)}\\
&\lesssim \|u\|_{{L}^{1}_t(\dot{B}^{1}_{\infty, 1})}\|u\|_{{L}^{\infty}_t(\dot{B}^{\frac{1}{2}}_{2, \infty})} 2^{q\left(2-\frac{3}{p}\right)}\sum_{k\geq
q-3} d_{k} 2^{-k}\\
&\lesssim d_{q} 2^{q\left(1-\frac{3}{p}\right)}\|u\|_{{L}^{1}_t(\dot{B}^{1}_{\infty, 1})} \|u\|_{{L}^{\infty}_t(\dot{B}^{\frac{1}{2}}_{2, \infty})}.
\end{split}
\end{equation*}

For the last term in \eqref{comma-1}, owing to the properties of spectral localization of the Littlewood-Paley
decomposition, we have $
I^{(4)}_q=\sum_{\vert
k-q\vert\leq 4}[\dot{\Delta}_{q},\dot{S}_{k-1}u^j]\dot{\Delta}_k\partial_j u.
$ Then we deduce from
 Lemma \ref{lem-commutator-1} and \eqref{S2eq3} that
\begin{equation*}
\begin{split}
\|I^{(4)}_q\|_{L^1_t(L^p)}
&\lesssim
\sum_{|k-q|\leq4} 2^{k-q}
\|\dot{\Delta}_{k}u\|_{L^1_t(L^{\infty})}
\|\dot{S}_{k-1}\nabla u\|_{L^\infty_t(L^p)}\\
&\lesssim \|u\|_{{L}^{1}_t(\dot{B}^{1}_{\infty, 1})} \|u\|_{{L}^{\infty}_t(\dot{B}^{\frac{1}{2}}_{2, \infty})}
\sum_{|k-q|\leq4}2^{-q}d_{k}2^{\left(2-\frac{3}{p}\right)k}\\
&\lesssim  d_{q} 2^{q\left(1-\frac{3}{p}\right)}\|u\|_{{L}^{1}_t(\dot{B}^{1}_{\infty, 1})} \|u\|_{{L}^{\infty}_t(\dot{B}^{\frac{1}{2}}_{2, \infty})}.
\end{split}
\end{equation*}

By summarizing the above estimates and using  \eqref{comma-1}, we arrive at \eqref{est-basic-2-121}.
\end{proof}

\begin{lem}\label{lem-density-1}
{\sl  Let $2\leq p <3$ and $3< \lambda <\infty$ with $ \frac{1}{\lambda}+\frac{1}{3}\geq \frac{1}{p}.$ Let
 $a\eqdefa\rho^{-1}-1\in  \widetilde{L}^\infty_t(\dot{B}^{\frac{3}{\lambda}}_{\lambda, 1}(\mathbb{R}^3))$ and $ f\in {L}^1_t(L^3(\mathbb{R}^3))$. Then we have
\begin{equation}\label{comma-18}
\begin{split}
\sum_{q\in\mathbb{Z}}2^{q\left(-1+\frac{3}{p}\right)}\|[\dot{\Delta}_{q}\mathbb{P}, \rho^{-1}] f\|_{L^1_t(L^p)}
\lesssim \|a\|_{\widetilde{L}^\infty_t(\dot{B}^{\frac{3}{\lambda}}_{\lambda, 1})}\|f\|_{{L}^1_t(L^3)}.
\end{split}
\end{equation}}
\end{lem}
\begin{proof} Similar to \eqref{comma-1}, we get, by applying
 Bony's decomposition \eqref{bony}, that
\begin{equation}\label{2.25}
\begin{split}
[\dot{\Delta}_{q}\mathbb{P}, \rho^{-1}] f&= [\dot{\Delta}_{q}\mathbb{P}, a]f= \dot{\Delta}_{q}\mathbb{P} R(a,f)+ \dot{\Delta}_{q}\mathbb{P}T_{f}a-T'_{\dot{\Delta}_q\mathbb{P}f}a-[\dot{\Delta}_{q}\mathbb{P},T_a]f \eqdefa \sum_{i=1}^4 J_q^{(i)}.
\end{split}
\end{equation}
Due to  $ \frac{1}{\lambda}+\frac{1}{3}\geq \frac{1}{p},$ we deduce from
 Lemma \ref{lem2.1} that
\begin{equation*}
\begin{split}
\|J_q^{(1)}\|_{L^1_t(L^p)}
&\lesssim 2^{3q\left(\frac{1}{3}+\frac{1}{\lambda}-\frac{1}{p}\right)}\sum_{k\geq q-3}\|\dot{\Delta}_k a\|_{L^{\infty}_t(L^{\lambda})}\|\widetilde{\dot{\Delta}}_kf\|_{L^1_t(L^{3})}\\
&\lesssim \|f\|_{{L}^1_t(L^3)}\|a\|_{\widetilde{L}^\infty_t(\dot{B}^{\frac{3}{\lambda}}_{\lambda, 1})}2^{3q\left(\frac{1}{3}+\frac{1}{\lambda}-\frac{1}{p}\right)}\sum_{k\geq q-3}d_k2^{-k\frac{3}{\lambda}}\\
&\lesssim d_{q}2^{q\left(1-\frac{3}{p}\right)}\|f\|_{{L}^1_t(L^3)}\|a\|_{\widetilde{L}^\infty_t(\dot{B}^{\frac{3}{\lambda}}_{\lambda, 1})}.
\end{split}
\end{equation*}
Along the same line, we infer
\begin{equation*}
\begin{split}
\|J_q^{(3)}\|_{L^p}
\lesssim
\sum_{k\geq q-3}\|{\Delta}_k a\dot{S}_{k+2}\dot{\Delta}_{q}f\|_{L^1_t(L^{p})}&
\lesssim
2^{3q\left(\frac{1}{3}+\frac{1}{\lambda}-\frac{1}{p}\right)}\|\dot{\Delta}_{q}f\|_{L^1_t(L^{3})} \sum_{k\geq q-3}\|\dot{\Delta}_k a\|_{L^\infty_t(L^{\lambda})}
\\&
\lesssim d_{q}2^{q\left(1-\frac{3}{p}\right)}
\|f\|_{{L}^1_t(L^3)}\|a\|_{\widetilde{L}^\infty_t(\dot{B}^{\frac{3}{\lambda}}_{\lambda, 1})}.
\end{split}
\end{equation*}

Notice that $\|\dot{S}_{k-1}f\|_{L^3}\leq\|f\|_{L^3},$
we deduce from  Lemma \ref{lem2.1} that
\begin{equation*}
\begin{split}
\|J_q^{(2)}\|_{L^1_t(L^p)}&\lesssim\sum_{|q-k|\leq4} \|\dot{S}_{k-1}f\|_{L^1_t(L^{3})}\|\dot{\Delta}_k a\|_{L^\infty_t(L^{\frac{3p}{3-p}})}\\
&\lesssim \|a\|_{\widetilde{L}^\infty_t(\dot{B}^{\frac{3}{\lambda}}_{\lambda, 1})}\|f\|_{{L}^1_t(L^3)}
\sum_{|q-k|\leq 4}d_k 2^{k\left(1-\frac{3}{p}\right)} \lesssim d_{q}2^{q\left(1-\frac{3}{p}\right)}\|f\|_{{L}^1_t(L^3)}\|a\|_{\widetilde{L}^\infty_t(\dot{B}^{\frac{3}{\lambda}}_{\lambda, 1})}.
\end{split}
\end{equation*}

Finally due to
\begin{align*}
\|\nabla \dot{S}_{k-1} a\|_{L^\infty_t(L^{\frac{3p}{3-p}})}\lesssim  \|a\|_{\widetilde{L}^\infty_t(\dot{B}^{\frac{3}{\lambda}}_{\lambda, 1})}
 \sum_{\ell\leq k-2}d_{\ell} 2^{\left(2-\frac{3}{p}\right)\ell}\lesssim d_k2^{\left(2-\frac{3}{p}\right)k} \|a\|_{\widetilde{L}^\infty_t(\dot{B}^{\frac{3}{\lambda}}_{\lambda, 1})},
 \end{align*}
we find
\begin{equation*}
\begin{split}
\|J_q^{(4)}\|_{L^1_t(L^p)}&\lesssim \sum_{|q-k|\leq4} 2^{-q}\|\dot{\Delta}_{k}f\|_{L^1_t(L^3)}\|\nabla \dot{S}_{k-1} a\|_{L^\infty_t(L^{\frac{3p}{3-p}})}\\
&\lesssim \|f\|_{L^1_t(L^3)} \|a\|_{\widetilde{L}^\infty_t(\dot{B}^{\frac{3}{\lambda}}_{\lambda, 1})} \sum_{|q-k|\leq4}d_k2^{-q} 2^{\left(2-\frac{3}{p}\right)k}\\
&\lesssim d_{q}2^{q\left(1-\frac{3}{p}\right)} \|f\|_{{L}^1_t(L^3)}\|a\|_{\widetilde{L}^\infty_t(\dot{B}^{\frac{3}{\lambda}}_{\lambda, 1})}.
\end{split}
\end{equation*}

By summarizing the above estimates and using \eqref{2.25}, we obtain
\begin{equation*}
\begin{split}
&\|[\dot{\Delta}_{q}\mathbb{P}, \rho^{-1}] f\|_{L^1_t(L^p)}\lesssim  d_{q}2^{q\left(1-\frac{3}{p}\right)} \|f\|_{{L}^1_t(L^3)}\|a\|_{\widetilde{L}^\infty_t(\dot{B}^{\frac{3}{\lambda}}_{\lambda, 1})},
\end{split}
\end{equation*}
 which ensures \eqref{comma-18}. This completes the proof of the lemma.
\end{proof}

\renewcommand{\theequation}{\thesection.\arabic{equation}}
\setcounter{equation}{0}

\section{The proof of Theorems \ref{thmmain-global}} \label{Sect3}


Motivated by Section 3 of \cite{Zhang2020}, we shall first derive the {\it a priori} estimates which will be used in the proof of
Theorem \ref{thmmain-global}. Then we present the proof of Theorems \ref{thmmain-global} at the end of this section.
 Toward this, we first present the following basic energy estimates for the linearized equation of \eqref{1.2}.

\begin{prop}\label{prop-basicE-1}
{\sl Let $(\rho,\,u)$ be a smooth solution of the system \eqref{1.2} on $[0, T^{\ast})$, for any $j\in\Z,$ $(u_j,\,\nabla\Pi_j)$ solves
\begin{equation}\label{model-3d-freq-1}
\begin{cases}
\rho(\pa_t u_j +u\cdot\nabla u_j) - (\Delta u_j-\grad\Pi_j)=0\quad \mbox{in} \quad \mathbb{R}^+\times \mathbb{R}^3, \\
 \dv\, u_j = 0, \\
u_j|_{t=0}= \dot{\Delta}_ju_0.
\end{cases}
\end{equation}
Then under the assumptions of \eqref{bdd-density-assum-1} and \eqref{smallness-u-1}, there hold \eqref{bdd-density-res-1} and
\begin{equation}\label{est-basic-j-1-1}
\begin{split}
&\|u_j\|_{L^\infty_T(L^2)}+ \|\nabla\,u_j\|_{L^2_T(L^2)}\lesssim \|\dot{\Delta}_ju_0\|_{L^2},\\
&\|\nabla\,u_j\|_{L_T^{\infty}(L^2)}+\|(\partial_tu_j,\,\nabla^2u_j,\,\grad\Pi_j)\|_{L^2_T(L^2)} \lesssim 2^{j} \|\dot{\Delta}_ju_0\|_{L^2},\quad \forall\,  j\in \mathbb{Z},
\end{split}
\end{equation}
\begin{equation}\label{est-basic-2-20eeb-app}
\|t^{-\alpha}\nabla u_j\|_{L^2_T(L^2)}
\lesssim
2^{2j\alpha }\|\dot\Delta_ju_0\|_{L^2} \quad (\forall\,\,\alpha\in[0,\frac{1}{2}),\,\, j\in \mathbb{Z}),
\end{equation}
and
\begin{equation}\label{est-basic-2-16}
\begin{split}
& \|u\|_{{L}_T^{\infty}(\dot{B}^{\frac{1}{2}}_{2, \infty})}+\|u\|_{\widetilde{L}^2_T(\dot{B}^{\frac{3}{2}}_{2, \infty})}\lesssim  \|u_0\|_{\dot{B}^{\frac{1}{2}}_{2, \infty}}.
\end{split}
\end{equation}

If in addition, $u_0\in \dot{B}^{\frac{1}{2}}_{2, r}$ for $1\leq r \leq +\infty$, then there holds
\begin{equation}\label{est-basic-2-18}
\begin{split}
& \|u\|_{\widetilde{L}_T^{\infty}(\dot{B}^{\frac{1}{2}}_{2, r})}+\|u\|_{\widetilde{L}^2_T(\dot{B}^{\frac{3}{2}}_{2, r})}
\lesssim
\|u_0\|_{\dot{B}^{\frac{1}{2}}_{2, r}}.
\end{split}
\end{equation}
In particular, one has
\begin{equation}\label{est-basic-2-18-a1}
\begin{split}
& \|u\|_{\widetilde{L}_T^{\infty}(\dot{H}^{\frac{1}{2}})}+\|u\|_{{L}^2_T(\dot{H}^{\frac{3}{2}})}+\|t^{-\frac{1}{4}}\nabla u\|_{L^2_T(L^2)}\lesssim  \|u_0\|_{\dot{H}^{\frac{1}{2}}}.
\end{split}
\end{equation}}
\end{prop}

\begin{proof}
We first deduce from the classical theory of transport equation and \eqref{bdd-density-assum-1} that
there holds \eqref{bdd-density-res-1} for $0\leq t <  T^{\ast}$.

Since  $(u_j,\,\nabla\Pi_j)$ solves \eqref{model-3d-freq-1}, we deduce from the classical uniqueness result of local smooth solution to \eqref{1.2} that
\begin{equation}\label{identity-1}
\begin{split}
u=\sum_{j \in \mathbb{Z}}u_j \andf \nabla\,\Pi=\sum_{j \in \mathbb{Z}}\nabla\,\Pi_j \quad \mbox{in} \quad \mathcal{S}'_h.
\end{split}
\end{equation}

By taking $L^2$ inner product of the momentum equation of \eqref{model-3d-freq-1} with $u_j$ and using the
transport equation of \eqref{1.2}, we find
\begin{equation*}\label{est-basic-1-1}
\begin{split}
\frac{1}{2}\frac{d}{dt}\int_{\mathbb{R}^3}\rho|u_j|^2\,dx+\|\nabla\,u_j\|_{L^2}^2=0,.
\end{split}
\end{equation*}
By integrating the above equation over $[0,t]$ for $t\in [0, T^{\ast}),$ we find
\begin{equation*}\label{est-basic-1-2}
\frac{1}{2}\|\sqrt{\rho}\,u_j(t)\|_{L^2}^2+\|\nabla\,u_j\|_{L^2_t(L^2)}^2=\frac{1}{2}\|\sqrt{\rho_0}\,\dot{\Delta}_ju_0\|_{L^2}^2,
\end{equation*}
from which, \eqref{bdd-density-assum-1} and \eqref{bdd-density-res-1}, we deduce that there exists a positive constant $c_1$ such that
\begin{equation}\label{est-basic-1-3}
\begin{split}
\|u_j\|_{L^\infty_t(L^2)}+c_1\|\nabla\,u_j\|_{L^2_t(L^2)}\lesssim \|\dot{\Delta}_ju_0\|_{L^2}.
\end{split}
\end{equation}
In particular, one has
\begin{equation}\label{est-basic-1-4}
\begin{split}
\|u_j\|_{L^\infty_t(L^2)}+c_1\|\nabla\,u_j\|_{L^2_t(L^2)} \lesssim 2^{-\frac{j}{2}}\|u_0\|_{\dot{B}^{\frac{1}{2}}_{2, \infty}}.
\end{split}
\end{equation}

Whereas by taking $L^2$ inner product of the momentum equation of \eqref{model-3d-freq-1} with $\partial_tu_j,$ we find
\begin{equation*}\label{est-basic-2-1}
\begin{split}
\frac{1}{2}\frac{d}{dt}\|\nabla\,u_j\|_{L^2}^2+\|\sqrt{\rho}\,\partial_tu_j\|_{L^2}^2=-\int_{\mathbb{R}^3}(\rho\,u\cdot\nabla u_j)\cdot \partial_t u_j\,dx \lesssim \|u\cdot\nabla u_j\|_{L^2}\|\sqrt{\rho}\,\partial_tu_j\|_{L^2},
\end{split}
\end{equation*}
from which and \eqref{est-lem-prod-1}, we infer
\begin{equation}\label{est-basic-2-2}
\begin{split}
\frac{1}{2}\frac{d}{dt}\|\nabla\,u_j\|_{L^2}^2+\|\sqrt{\rho}\,\partial_tu_j\|_{L^2}^2\lesssim \|u\|_{\dot{B}^{\frac{1}{2}}_{2, \infty}}\|\nabla^2\,u_j\|_{L^2}\|\sqrt{\rho}\,\partial_tu_j\|_{L^2}.
\end{split}
\end{equation}
While in view of \eqref{model-3d-freq-1}, we write
\begin{equation}\label{eqns-basic-2-3}
\begin{cases}
\Delta u_j-\grad\Pi_j=\rho(\pa_t u_j +u\cdot\nabla u_j),\\
\dive u_j=0,
\end{cases}
\end{equation}
from which and the classical theory on Stokes operator, we infer
\begin{equation*}\label{est-basic-2-4}
\begin{split}
\|\Delta u_j\|_{L^2}+\|\grad\Pi_j\|_{L^2}\lesssim \|\sqrt{\rho}\,\partial_tu_j\|_{L^2}+\|u\cdot\nabla u_j\|_{L^2}.
\end{split}
\end{equation*}
By applying \eqref{est-lem-prod-1} once again, we obtain
\begin{equation}\label{est-basic-2-5}
\begin{split}
\|(\nabla^2 u_j,\,\grad\Pi_j)\|_{L^2}\leq C_1(\|\sqrt{\rho}\,\partial_tu_j\|_{L^2}+\|u\|_{\dot{B}^{\frac{1}{2}}_{2, \infty}}\|\nabla^2\,u_j\|_{L^2}).
\end{split}
\end{equation}

By combining  \eqref{est-basic-2-2} with \eqref{est-basic-2-5}, we find
\begin{equation*}\label{est-basic-2-6}
\begin{split}
\frac{1}{2}\frac{d}{dt}\|\nabla\,u_j\|_{L^2}^2+&2\mathfrak{c}_0\|(\sqrt{\rho}\,\partial_tu_j, \nabla^2 u_j, \grad\Pi_j)\|_{L^2}^2\\
&\leq C_2\Bigl(\|u\|_{\dot{B}^{\frac{1}{2}}_{2, \infty}}\|\nabla^2\,u_j\|_{L^2}\|\sqrt{\rho}\,\partial_tu_j\|_{L^2}+\|u\|_{\dot{B}^{\frac{1}{2}}_{2, \infty}}^2\|\nabla^2\,u_j\|_{L^2}^2\Bigr).
\end{split}
\end{equation*}
Applying Young's inequality yields
\begin{equation}\label{est-basic-2-7}
\begin{split}
&\frac{d}{dt}\|\nabla\,u_j\|_{L^2}^2+2\mathfrak{c}_0\|(\sqrt{\rho}\,\partial_tu_j, \nabla^2 u_j, \grad\Pi_j)\|_{L^2}^2\leq \mathfrak{C}_0 \|u\|_{\dot{B}^{\frac{1}{2}}_{2, \infty}}^2\|\na^2\,u_j\|_{L^2}^2,
\end{split}
\end{equation}
Let us denote
\begin{equation}\label{small-assump-u-1}
T_1\eqdefa \sup\Bigl\{\  t\in (0, T^{\ast}):\quad \|u\|_{L^{\infty}_{t}(\dot{B}^{\frac{1}{2}}_{2, \infty})}\leq (\frac{\mathfrak{c}_0}{\mathfrak{C}_0})^{\frac{1}{2}} \Bigr\}.
\end{equation}
We will claim that $T_1=T^{\ast}$. 

In fact, by contradiction, if $T_1<T^{\ast}$, then for any $t\in [0, T_1]$, we deduce from \eqref{est-basic-2-7} that
\begin{equation}\label{est-basic-2-8}
\begin{split}
&\frac{d}{dt}\|\nabla\,u_j\|_{L^2}^2+\mathfrak{c}_0\|(\sqrt{\rho}\,\partial_tu_j, \nabla^2 u_j, \grad\Pi_j)\|_{L^2}^2\leq 0.
\end{split}
\end{equation}
By integrating the above inequality over $[0,t]$ for $t\leq T_1,$ one has
\begin{equation}\label{est-basic-2-10}
\begin{split}
\|\nabla\,u_j\|_{L_t^{\infty}(L^2)}+\|(\partial_tu_j,\,\nabla^2u_j,\,\grad\Pi_j)\|_{L^2_t(L^2)}\lesssim 2^{j}\|\dot{\Delta}_ju_0\|_{L^2},
\end{split}
\end{equation}
which implies
\begin{equation}\label{est-basic-2-11}
\begin{split}
&\|\nabla\,u_j\|_{L_t^{\infty}(L^2)}+\|(\partial_tu_j, \nabla^2 u_j, \grad\Pi_j)\|_{L^2_t(L^2)}\leq C 2^{\frac{j}{2}} \|u_0\|_{\dot{B}^{\frac{1}{2}}_{2, \infty}}.
\end{split}
\end{equation}

On the other hand, thanks to \eqref{identity-1}, we deduce from Lemma \ref{lem2.1} that
\begin{equation*}\label{est-basic-2-12}
\begin{split}
\|\dot{\Delta}_j u\|_{L_t^{\infty}(L^2)}+\|\nabla \dot{\Delta}_j u\|_{L^2_t(L^2)}&\lesssim \sum_{j'\in \mathbb{Z}}(\|\dot{\Delta}_j u_{j'}\|_{L_t^{\infty}(L^2)}+ \|\nabla \dot{\Delta}_j u_{j'}\|_{L^2_t(L^2)})\\
&\lesssim \sum_{j'\geq j}\bigl(\|\dot{\Delta}_j u_{j'}\|_{L_t^{\infty}(L^2)}+\|\nabla \dot{\Delta}_j u_{j'}\|_{L^2_t(L^2)}\bigr)\\
&\qquad+2^{-j}\sum_{j'\leq j}\bigl(\|\nabla\dot{\Delta}_j u_{j'}\|_{L_t^{\infty}(L^2)}+ \|\nabla^2 \dot{\Delta}_j u_{j'}\|_{L^2_t(L^2)}\bigr),
\end{split}
\end{equation*}
from which, we infer
\begin{equation}\label{est-basic-2-13}
\begin{split}
 \|\dot{\Delta}_j u\|_{L_t^{\infty}(L^2)}+\|\nabla \dot{\Delta}_j u\|_{L^2_t(L^2)}
\lesssim &\sum_{j'\geq j}\bigl(\|u_{j'}\|_{L_t^{\infty}(L^2)}+\|\nabla u_{j'}\|_{L^2_t(L^2)}\bigr)\\
&+2^{-j}\sum_{j'\leq j}(\|\nabla\,u_{j'}\|_{L_t^{\infty}(L^2)}+ \|\nabla^2u_{j'}\|_{L^2_t(L^2)}).
\end{split}
\end{equation}
By substituting \eqref{est-basic-1-4} and \eqref{est-basic-2-11} into \eqref{est-basic-2-13}, we obtain
\begin{equation*}\label{est-basic-2-15}
\begin{split}
\|\dot{\Delta}_j u\|_{L_t^{\infty}(L^2)}+\|\nabla \dot{\Delta}_j u\|_{L^2_t(L^2)}\lesssim &\sum_{j'\geq j}2^{-\frac{j'}{2}}\|u_0\|_{\dot{B}^{\frac{1}{2}}_{2, \infty}}+2^{-j}\sum_{j'\leq j}2^{\frac{j'}{2}} \|u_0\|_{\dot{B}^{\frac{1}{2}}_{2, \infty}}\\
\lesssim &  2^{-\frac{j}{2}} \|u_0\|_{\dot{B}^{\frac{1}{2}}_{2, \infty}}.
\end{split}
\end{equation*}
As a result, it comes out
\begin{equation}\label{est-basic-2-15-aa}
\begin{split}
& \|u\|_{{L}_t^{\infty}(\dot{B}^{\frac{1}{2}}_{2, \infty})}+\|u\|_{\widetilde{L}^2_t(\dot{B}^{\frac{3}{2}}_{2, \infty})}\leq \mathfrak{C}_1\|u_0\|_{\dot{B}^{\frac{1}{2}}_{2, \infty}} \quad \forall\, t \in [0, T_1].
\end{split}
\end{equation}
 Then   if we take $\frak{c}$ in \eqref{smallness-u-1} to be so small  that $\|u_0\|_{\dot{B}^{\frac{1}{2}}_{2, \infty}} \leq \frak{c}\leq \frac{1}{2\mathfrak{C}_1} \min\{\Bigl(\frac{\mathfrak{c}_0}{\mathfrak{C}_0}\Bigr)^{\frac{1}{2}}, \frac{1}{C_1}\},$ we deduce from \eqref{est-basic-2-15-aa} that 
 \begin{equation}\label{small-assump-u-2aa}
 \|u\|_{L_{T_1}^{\infty}(\dot{B}^{\frac{1}{2}}_{2, \infty})}\leq \mathfrak{C}_1 \mathfrak{c}\leq 
   \frac{1}{2}\min\{\Bigl(\frac{\mathfrak{c}_0}{\mathfrak{C}_0}\Bigr)^{\frac{1}{2}}, \frac{1}{C_1}\},
   \end{equation}
   which contradicts with \eqref{small-assump-u-1}. This in turn shows that  $T_1=T^{\ast}$, and there hold  \eqref{est-basic-j-1-1}, \eqref{est-basic-2-16}. Furthermore, it follows from \eqref{est-basic-2-5} and \eqref{small-assump-u-2aa} that
\begin{equation}\label{est-basic-2-5aa}
\begin{split}
\|\nabla^2 u_j(t)\|_{L^2}^2+\|\grad\Pi_j(t)\|_{L^2}^2\leq C\|\sqrt{\rho}\,\partial_tu_j(t)\|_{L^2}^2\quad \forall \ t \in [0, T^{\ast}).
\end{split}
\end{equation}

For $\alpha\in(0,\frac{1}{2}),$ thanks to \eqref{est-basic-j-1-1}, we have
\begin{equation*} 
\begin{split}
\int_0^t{\tau}^{-2\alpha}\|\nabla u_j\|_{L^2}^2d\tau
&\leq
\int_0^{t_0}{\tau}^{-2\alpha}\|\nabla u_j\|_{L^2}^2d\tau
+\int_{t_0}^t {\tau}^{-2\alpha}\|\nabla u_j\|_{L^2}^2d\tau
\\&
\lesssim
{t_0}^{1-2\alpha}
2^{2j}\|\dot\Delta_ju_0\|_{L^2}^2
+{t_0}^{-2\alpha}\|\dot\Delta_ju_0\|_{L^2}^2,
\end{split}
\end{equation*}
then for $ t_0\approx 2^{-2j},$ we reach \eqref{est-basic-2-20eeb-app}.

If in addition, $u_0\in \dot{B}^{\frac{1}{2}}_{2, r}$ for $1\leq r \leq +\infty$, then
 by inserting \eqref{est-basic-1-3} and \eqref{est-basic-2-10} into \eqref{est-basic-2-13}, we obtain
\begin{equation*}\label{est-basic-2-17}
\begin{split}
 \|\dot{\Delta}_j u\|_{L_t^{\infty}(L^2)}&+\|\nabla \dot{\Delta}_j u\|_{L^2_t(L^2)}\lesssim  \sum_{j'\geq j}\|\dot{\Delta}_ju_0\|_{L^2}+2^{-j}\sum_{j'\leq j}\|\nabla\dot{\Delta}_ju_0\|_{L^2}\\
&\lesssim \sum_{j'\geq j}c_{j', r}2^{-\frac{j'}{2}}\|u_0\|_{\dot{B}^{\frac{1}{2}}_{2, r}}+2^{-j}\sum_{j'\leq j}c_{j', r}2^{\frac{j'}{2}}\|u_0\|_{\dot{B}^{\frac{1}{2}}_{2, r}}\lesssim  c_{j, r}2^{-\frac{j}{2}}\|u_0\|_{\dot{B}^{\frac{1}{2}}_{2, r}},
\end{split}
\end{equation*}
which leads to \eqref{est-basic-2-18}. 

Furthermore, taking $\varepsilon\in(0,\frac{1}{4})$, we get from \eqref{est-basic-2-20eeb-app} that
\begin{equation*} 
\begin{split}
&\|{\tau}^{-(\frac{1}{2}-\varepsilon)}\nabla u_j\|_{L^2_t(L^2)}
\lesssim
2^{j(\frac{1}{2}-2\varepsilon)}c_{j,2}\|u_0\|_{\dot H^{\frac{1}{2}}},\quad \|{\tau}^{-\varepsilon}\nabla u_j\|_{L^2_t(L^2)}
\lesssim
2^{j(2\varepsilon-\frac{1}{2})}c_{j,2}\|u_0\|_{\dot H^{\frac{1}{2}}},
\end{split}
\end{equation*}
which implies
\begin{equation*} 
\begin{split}
\|{\tau}^{-\frac{1}{4}}\nabla u\|_{L^2_t(L^2)}^2
&\leq
2\sum_{k \in \mathbb{Z}}\|{\tau}^{-\varepsilon}\nabla u_k\|_{L^2_t(L^2)}\sum_{j\le k}
\|{\tau}^{-(\frac{1}{2}-\varepsilon)}\nabla u_j\|_{L^2_t(L^2)}
\\&
\lesssim
\|u_0\|_{\dot H^{\frac{1}{2}}}^2 \sum_{k \in \mathbb{Z}}2^{k(2\varepsilon-\frac{1}{2})}c_{k,2}
\sum_{j\le k}
2^{j(\frac{1}{2}-2\varepsilon)}c_{j,2}
\\&
\lesssim
\|u_0\|_{\dot H^{\frac{1}{2}}}^2 \sum_{k \in \mathbb{Z}}c_{k,2}^2\lesssim
\|u_0\|_{\dot H^{\frac{1}{2}}}^2.
\end{split}
\end{equation*}
This completes the proof of Proposition \ref{prop-basicE-1}.
\end{proof}

\begin{rmk}\label{rmk-small-const-dep-1}
From the proof of Proposition \ref{prop-basicE-1}, we find that all the constants in the proof of Proposition \ref{prop-basicE-1} depend only on the upper and lower bounds of the initial density $\rho_0$, which ensures that the small positive constant $\frak{c}$ in \eqref{smallness-u-1} depends only on $c_0$ and $C_0$, bounds of the initial density $\rho_0$.
\end{rmk}

\begin{prop}\label{prop-time-energy-j-1}
{\sl  Let $u_0\in\dot H^{\f12}(\mathbb{R}^3)$ with $\dv\,u_0=0$. Then  under the assumption in Proposition \ref{prop-basicE-1},
  for any $\,j\in \mathbb{Z}$, there hold
\ben
&&\|t^{\frac{1}{2}}\nabla\,u_j\|_{L^\infty_T(L^2)}+\|t^{\frac{1}{2}}(\partial_tu_j,\,\nabla^2\,u_j\,\nabla\,\Pi_j)\|_{L^2_T(L^2)}\lesssim \|\dot{\Delta}_ju_0\|_{L^2},\label{est-basic-2-20eea}\\
&&\|t^{\frac{1}{2}}(\partial_t u_j,\nabla^2 u_j, \nabla\Pi_j)\|_{L^\infty_{T}(L^2)}+\|t^{\frac{1}{2}}\nabla
\partial_t u_j\|_{L^2_{T}(L^2)} \lesssim  2^j\|\dot{\Delta}_ju_0\|_{L^2},\label{j-est-basic-2-32}\\
&&\|t^{\frac{1}{2}}D_t u_j\|_{L^\infty_{T}(L^2)}+\|t^{\frac{1}{2}}\nabla
D_t u_j\|_{L^2_{T}(L^2)}+\|t^{\frac{1}{2}}(\nabla^2 u_j,\nabla \Pi_j)\|_{L^2_{T}(L^6)}\lesssim 2^{j}\|\dot{\Delta}_ju_0\|_{L^2},\label{est-basic-2-20eec}\\
&&\|u\|_{L^2_T(L^{\infty})}+\|t^{\frac{1}{4}}(\nabla^2\,u, \nabla\Pi, \partial_tu)\|_{L^2_T(L^{2})}+\|t^{\frac{1}{2}}\nabla^2\,u\|_{L^2_T(\dot H^{\f12})}\lesssim \|u_0\|_{\dot H^{\f12}}.\label{est-basic-2-20eed}
\een
Henceforth we always denote $D_tu\eqdefa \partial_t+u\cdot\nabla $ to be the material derivative.

 If in addition $u_0\in \dot{B}^{\frac{1}{2}}_{2, r}$ with  $1\leq r \leq +\infty,$ there holds
 \beq\label{S3eq5}
\|t^{\frac{1}{2}} u\|_{\widetilde{L}^\infty_T(\dot{B}^{\frac{3}{2}}_{2, r})}+\|t^{\frac{1}{2}}\partial_t u\|_{\widetilde{L}^2_T(\dot{B}^{\frac{1}{2}}_{2, r})}
\lesssim \|u_0\|_{\dot{B}^{\frac{1}{2}}_{2, r}}.
\eeq
}
\end{prop}

\begin{proof}
It is easy to observe from  \eqref{est-basic-2-8} that
\begin{equation*}\label{est-basic-2-19}
\begin{split}
&\frac{d}{dt}\|t^{\frac{1}{2}}\nabla\,u_j\|_{L^2}^2+\mathfrak{c}_0 \|t^{\frac{1}{2}}(\sqrt{\rho}\,\partial_tu_j,\,\nabla^2 u_j,\,\grad\Pi_j)\|_{L^2}^2 \leq \|\nabla\,u_j\|_{L^2}^2.
\end{split}
\end{equation*}
By integrating the above inequality over $[0,T]$ and using  \eqref{est-basic-1-3}, we find
\begin{equation}\label{est-basic-2-20aa}
\begin{split}
&\|t^{\frac{1}{2}}\nabla\,u_j\|_{L^\infty_T(L^2)}+\|t^{\frac{1}{2}}(\partial_tu_j,\,\nabla^2\,u_j, \,\nabla\,\Pi_j)\|_{L^2_T(L^2)}\lesssim \|\nabla\,u_j\|_{L^2_T(L^2)} \lesssim \|\dot{\Delta}_ju_0\|_{L^2},
\end{split}
\end{equation}
and  \eqref{est-basic-2-20eea} follows.

Whereas by applying $\partial_t$ to the $u_j$ equation of \eqref{model-3d-freq-1} and then taking the $L^2$
inner product of the resulting equation with $\pa_t u_j,$ we find
\begin{equation}\label{j-est-basic-2-21}
\begin{split}
\frac{1}{2}\frac{d}{dt}\|\sqrt{\rho} \partial_t u_j\|_{L^2}^2+\|\nabla
\partial_t u_j\|_{L^2}^2
=&-\int_{\R^3}\rho_t\partial_t u_j \cdot\bigl(\partial_t u_j +u\cdot\nabla u_j\bigr)\, dx\\
&-\int_{\R^3}\rho \partial_t u_j \cdot (u_t \cdot\nabla u_j)\, dx\eqdefa I_j+II_j+III_j.
\end{split}
\end{equation}
For $I_j$, we get, by using the transport equation of \eqref{1.2}, that
\begin{equation*}
\begin{split}
&I_j=-\int_{\R^3}\rho_t|\partial_t u_j|^2\,dx=\int_{\R^3}\nabla\cdot(\rho\,u)|\partial_t u_j|^2\,dx=-2\int_{\R^3}\rho\,(u\cdot \nabla)\partial_t u_j\cdot \partial_t u_j \,dx,
\end{split}
\end{equation*}
from which and  \eqref{est-lem-prod-1}, we infer
\begin{equation}\label{j-est-basic-2-22}
\begin{split}
&|I_j|\lesssim \|\rho\|_{L^\infty}\|\nabla\partial_t u_j\|_{L^2} \|u\,\partial_t u_j\|_{L^{2}}\lesssim \|u\|_{\dot{B}^{\frac{1}{2}}_{2, \infty}}\|\nabla\partial_t u_j\|_{L^2}^2.
\end{split}
\end{equation}
Similarly for $II_j$, there holds
\begin{align*}
II_j=&-\int_{\R^3}\rho_t\partial_t u_j  \cdot (u\cdot\nabla u_j)\, dx=\int_{\R^3}\nabla\cdot(\rho\,u)\,\bigl[\partial_t u_j  \cdot (u\cdot\nabla u_j)\bigr]\,dx\\
=&-\int_{\R^3} \rho\,u \cdot\bigl[\nabla\partial_t u_j  \cdot (u\cdot\nabla u_j)\bigr]\,dx-\int_{\R^3} \rho\,u\cdot\bigl[\partial_t u_j  \cdot (\nabla\,u\cdot\nabla u_j)\bigr]\,dx\\
&-\int_{\R^3} \rho\,u\cdot\bigl[\partial_t u_j  \cdot (u\cdot\nabla^2 u_j)\bigr]\,dx\eqdefa \sum_{k=1}^3II^{(k)}_{j}.
\end{align*}
It follows from Lemmas \ref{lem-product-1-1} and \ref{lem-product-uu-2aaa} that
\begin{equation*}
\begin{split}
|II^{(1)}_{j}|&\lesssim\|\rho\|_{L^\infty}\|\nabla\partial_t u_j\|_{L^{2}} \|u\otimes u\otimes\nabla u_j\|_{L^{2}}\\
&\lesssim \|\nabla\partial_t u_j\|_{L^{2}}\|u\otimes u\|_{\dot{B}^{\frac{1}{2}}_{2, \infty}} \|\nabla^2 u_j\|_{L^{2}}\lesssim \|u\|_{\dot{B}^{\frac{1}{2}}_{2, \infty}}\|u\|_{\dot{B}^{\frac{3}{2}}_{2, \infty}} \|\nabla\partial_t u_j\|_{L^{2}}\|\nabla^2 u_j\|_{L^{2}},\\
 |II^{(2)}_{j}|&\lesssim \|\rho\|_{L^\infty} \|u\,\partial_t u_j\|_{L^{2}} \|\nabla\,u\,\nabla u_j\|_{L^{2}}\lesssim  \|u\|_{\dot{B}^{\frac{1}{2}}_{2, \infty}}\|u\|_{\dot{B}^{\frac{3}{2}}_{2, \infty}} \|\nabla\partial_t u_j\|_{L^{2}}\|\nabla^2 u_j\|_{L^{2}},\\
|II^{(3)}_{j}|&\lesssim \|\rho\|_{L^\infty}  \|u\otimes u\|_{\dot{B}^{\frac{1}{2}}_{2, \infty}} \|\nabla\partial_t u_j\|_{L^{2}} \|\nabla^2 u_j\|_{L^{2}}\lesssim \|u\|_{\dot{B}^{\frac{1}{2}}_{2, \infty}}\|u\|_{\dot{B}^{\frac{3}{2}}_{2, \infty}} \|\nabla\partial_t u_j\|_{L^{2}}\|\nabla^2 u_j\|_{L^{2}},
\end{split}
\end{equation*}
so that there holds
\begin{equation}\label{j-est-basic-2-23}
\begin{split}
|II_{j}|&\lesssim \|u\|_{\dot{B}^{\frac{1}{2}}_{2, \infty}}\|u\|_{\dot{B}^{\frac{3}{2}}_{2, \infty}} \|\nabla\partial_t u_j\|_{L^{2}}\|\nabla^2 u_j\|_{L^{2}}.
\end{split}
\end{equation}
For $III_j$, one has
\begin{equation}\label{j-est-basic-2-24}
\begin{split}
&|III_j|\lesssim
\|\rho\|_{L^\infty} \|\partial_t u_j \|_{L^{6}}
\|u_t\|_{L^{2}}  \|\nabla u_j\|_{L^{3}}
\lesssim
\|\nabla\partial_t u_j \|_{L^{2}} \|u_t\|_{L^{2}}
\|\nabla u_j\|_{L^{2}}^{\frac{1}{2}} \|\nabla^2 u_j\|_{L^{2}}^{\frac{1}{2}}.
\end{split}
\end{equation}
By inserting the  estimates \eqref{j-est-basic-2-22}, \eqref{j-est-basic-2-23} and \eqref{j-est-basic-2-24} into \eqref{j-est-basic-2-21}, we obtain
\begin{equation*}\label{j-est-basic-2-25}
\begin{split}
&\frac{1}{2}\frac{d}{dt}\|\sqrt{\rho} \,\partial_t u_j\|_{L^2}^2+\|\nabla
\partial_t u_j\|_{L^2}^2 \leq C\|u\|_{\dot{B}^{\frac{1}{2}}_{2, \infty}}\|\nabla\partial_t u_j\|_{L^2}^2\\
&\qquad+C\|\nabla\partial_t u_j\|_{L^{2}}\Bigl(\|u\|_{\dot{B}^{\frac{1}{2}}_{2, \infty}}\|u\|_{\dot{B}^{\frac{3}{2}}_{2, \infty}}\|\nabla^2 u_j\|_{L^{2}}+\|u_t\|_{L^{2}}
\|\nabla u_j\|_{L^{2}}^{\frac{1}{2}} \|\nabla^2 u_j\|_{L^{2}}^{\frac{1}{2}}\Bigr).
\end{split}
\end{equation*}
Thanks to \eqref{smallness-u-1} and \eqref{est-basic-2-16}, we get, by applying Young's inequality and using \eqref{est-basic-2-5aa}, that
\begin{equation}\label{j-est-basic-2-26}
\begin{split}
\frac{d}{dt}\|\sqrt{\rho} \,\partial_t u_j\|_{L^2}^2+\|\nabla
\partial_t u_j\|_{L^2}^2
\lesssim &\|u\|_{\dot{B}^{\frac{1}{2}}_{2, \infty}}^2\|u\|_{\dot{B}^{\frac{3}{2}}_{2, \infty}}^2\|\sqrt{\rho} \,\partial_t u_j\|_{L^{2}}^2\\
&+\|u_t\|_{L^{2}}^2\|\nabla u_j\|_{L^{2}} \|\sqrt{\rho} \,\partial_t u_j\|_{L^{2}}.
\end{split}
\end{equation}
By multiplying  the above inequality by $t$, we find
\begin{equation*}
\begin{split}
&\frac{d}{dt}\|t^{\frac{1}{2}}  \sqrt{\rho}\, \partial_t u_j\|_{L^2}^2+\|t^{\frac{1}{2}}\nabla
\partial_t u_j\|_{L^2}^2 \leq \|\sqrt{\rho}\,\partial_t u_j\|_{L^{2}}^2\\
&\qquad+C\|u\|_{\dot{B}^{\frac{1}{2}}_{2, \infty}}^2\|u\|_{\dot{B}^{\frac{3}{2}}_{2, \infty}}^2\|t^{\frac{1}{2}}\sqrt{\rho}\,\partial_t u_j\|_{L^{2}}^2+C\|t^{\frac{1}{4}\,}u_t\|_{L^{2}}^2\|\nabla u_j\|_{L^{2}} \|t^{\frac{1}{2}}\sqrt{\rho} \,\partial_t u_j\|_{L^{2}},
\end{split}
\end{equation*}
from which, we infer
\begin{equation}\label{j-est-basic-2-27}
\begin{split}
\frac{d}{dt}\|t^{\frac{1}{2}}\sqrt{\rho}\, \partial_t u_j\|_{L^2}^2+\|t^{\frac{1}{2}}\nabla
\partial_t u_j\|_{L^2}^2
&\leq C\bigl(\|u\|_{\dot{B}^{\frac{1}{2}}_{2}}^2\|u\|_{\dot{B}^{\frac{3}{2}}_{2, \infty}}^2+\|t^{\frac{1}{4}}u_t\|_{L^{2}}^2\bigr)\|t^{\frac{1}{2}} \sqrt{\rho}\,\partial_t u_j\|_{L^{2}}^2 \\
&\qquad +   C\bigl(\|\partial_t u_j\|_{L^{2}}^2+ \|t^{\frac{1}{4}} u_t\|_{L^{2}}^2
\|\nabla u_j\|_{L^{2}}^2\bigr).
\end{split}
\end{equation}
Applying Gronwall's inequality gives rise to
\begin{equation}\label{j-est-basic-2-28}
\begin{split}
&\|t^{\frac{1}{2}}\partial_t u_j\|_{L^\infty_{t}(L^2)}^2+\|t^{\frac{1}{2}}\nabla
\partial_t u_j\|_{L^2_{t}(L^2)}^2 \\
&\leq C\bigl(\|\partial_t u_j\|_{L^2_{t}(L^{2})}^2+\|{t}^{\frac{1}{4}}\,u_{t}\|_{L^2_{t}(L^{2})}^2\|\nabla u_j\|_{L^\infty_{t}(L^{2})}^2\bigr) \\
&\qquad \times\exp\Bigl(C\bigl(\|u\|_{L^\infty_{t}(\dot{B}^{\frac{1}{2}}_{2, \infty})}^2\|u\|_{L^2_{t}(\dot{B}^{\frac{3}{2}}_{2, \infty})}^2+\|{t}^{\frac{1}{4}}\,u_{t}\|_{L^2_{t}(L^{2})}^2\bigr)\Bigr).
\end{split}
\end{equation}
It follows from  \eqref{est-basic-2-16} that
\begin{equation}\label{est-u-H12-infty-1}
\begin{split}
&\|u\|_{L^\infty_{T}(\dot{B}^{\frac{1}{2}}_{2, \infty})}^2
\|u\|_{L^2_{T}(\dot{B}^{\frac{3}{2}}_{2, \infty})}^2
\lesssim
\|u_0\|_{\dot{B}^{\frac{1}{2}}_{2, \infty}}^4\lesssim
\|u_0\|_{\dot{H}^{\frac{1}{2}}}^2.
\end{split}
\end{equation}
While for $\|{t}^{\frac{1}{4}}\,u_{t}\|_{L^2_{T}(L^{2})}$, we first computer
\begin{equation*}
\begin{split}
&\|{t}^{\frac{1}{4}}\,u_{t}\|_{L^2_T(L^{2})}^2=\int_{0}^Tt^{\frac{1}{2}}\int_{\mathbb{R}^3}\sum_{j\in \mathbb{Z}}\partial_tu_j\sum_{k\in \mathbb{Z}}\partial_tu_k\,dxdt\\
&\leq 2\sum_{j\in \mathbb{Z}}\sum_{k\leq j}\int_{0}^T\int_{\mathbb{R}^3}t^{\frac{1}{2}}\partial_tu_j\,\partial_tu_k\,dx\,dt\lesssim  \sum_{j\in \mathbb{Z}}\sum_{k\leq j}\|t^{\frac{1}{2}}\partial_tu_j\|_{L^2_T(L^{2})}\|\partial_tu_k\|_{L^2_T(L^{2})},
\end{split}
\end{equation*}
which together with \eqref{est-basic-2-10} and \eqref{est-basic-2-20aa} ensures that
\begin{equation}\label{est-basic-2-30}
\begin{split}
\|{t}^{\frac{1}{4}}\,u_{t}\|_{L^2_T(L^{2})}^2&\lesssim \sum_{j\in \mathbb{Z}}\sum_{k\leq j}\|\dot{\Delta}_ju_0\|_{L^2}\|\nabla\dot{\Delta}_ku_0\|_{L^2}\\
&\lesssim \|u_0\|_{\dot{H}^{\frac{1}{2}}}^2\sum_{j\in \mathbb{Z}}c_{j}2^{-\frac{j}{2}}\sum_{k\leq j}c_{k}2^{\frac{k}{2}}
\lesssim \|u_0\|_{\dot{H}^{\frac{1}{2}}}^2\sum_{j\in \mathbb{Z}}c_{j}^2\lesssim \|u_0\|_{\dot{H}^{\frac{1}{2}}}^2.
\end{split}
\end{equation}
As a consequence, we get from the momentum equation in \eqref{1.2} and \eqref{est-lem-prod-1} that
\begin{equation*}
\begin{split}
\|{t}^{\frac{1}{4}}\,(\nabla^2u, \nabla\Pi)\|_{L^2_T(L^{2})}&\lesssim\|{t}^{\frac{1}{4}}\,u_{t}\|_{L^2_T(L^{2})}+\|{t}^{\frac{1}{4}}\,u\cdot \nabla{u}\|_{L^2_T(L^{2})}\\
&\lesssim\|{t}^{\frac{1}{4}}\,u_{t}\|_{L^2_T(L^{2})}+\|u\|_{L^\infty_{T}(\dot{B}^{\frac{1}{2}}_{2, \infty})}\|{t}^{\frac{1}{4}} \nabla^2{u}\|_{L^2_T(L^{2})},
\end{split}
\end{equation*}
which along with \eqref{est-basic-2-16} ensures
\begin{equation}\label{est-basic-2-30-conseq}
\begin{split}
\|{t}^{\frac{1}{4}}\,(\nabla^2u, \nabla\Pi, u_t)\|_{L^2_T(L^{2})} \lesssim\|{t}^{\frac{1}{4}}\,u_{t}\|_{L^2_T(L^{2})} \lesssim \|u_0\|_{\dot{H}^{\frac{1}{2}}}.
\end{split}
\end{equation}
By substituting the above estimates and \eqref{est-basic-2-10} into \eqref{j-est-basic-2-28}, we obtain for any $t\leq T$
\begin{align*}
\|t^{\frac{1}{2}} \partial_t u_j\|_{L^\infty_t(L^2)}^2&+\|t^{\frac{1}{2}}\nabla
\partial_t u_j\|_{L^2_t(L^2)}^2\leq C(\|u_0\|_{\dot{H}^{\frac{1}{2}}}) \|\nabla\dot{\Delta}_ju_0\|_{L^2}^2 \with\\
&C(\|u_0\|_{\dot{H}^{\frac{1}{2}}})\eqdefa C\bigl( 1+\|u_0\|_{\dot{H}^{\frac{1}{2}}}^2\bigr)\,\exp\left(C\|u_0\|_{\dot{H}^{\frac{1}{2}}}^2\right).
\end{align*}
So that the inequality \eqref{j-est-basic-2-32} holds according to \eqref{est-basic-2-5aa}.

Thanks to \eqref{est-basic-2-20aa} and \eqref{j-est-basic-2-32}, we deduce that
\begin{align*}
&\|t^{\frac{1}{2}}D_t u_j\|_{L^\infty_{T}(L^2)}+\|t^{\frac{1}{2}}\nabla
D_t u_j\|_{L^2_{T}(L^2)}\\
&\lesssim \|t^{\frac{1}{2}}\partial_t u_j\|_{L^\infty_{T}(L^2)}+\|t^{\frac{1}{2}}u\cdot\nabla u_j\|_{L^\infty_{T}(L^2)}+\|t^{\frac{1}{2}}\nabla
\partial_t u_j\|_{L^2_{T}(L^2)}+\|t^{\frac{1}{2}}\nabla(u\cdot\nabla) u_j\|_{L^2_{T}(L^2)}\\
&\lesssim 2^{j}\|\dot{\Delta}_ju_0\|_{L^2}+\|u\|_{L^{\infty}_{T}(L^{3, \infty})}\|t^{\frac{1}{2}}\nabla u_j\|_{L^\infty_{T}(L^{6, 2})}+\|u\|_{L^{\infty}_{T}(L^{3})}\|t^{\frac{1}{2}}\nabla^2 u_j\|_{L^{2}_{T}(L^{6})}\\
&\qquad +\|\nabla\,u\|_{L^2_{T}(L^{3, \infty})}\|t^{\frac{1}{2}}\nabla u_j\|_{L^{\infty}_{T}(L^{6, 2})}\\
&\lesssim 2^{j}\|\dot{\Delta}_ju_0\|_{L^2}+\|u\|_{L^{\infty}_{t}(\dot{B}^{\frac{1}{2}}_{2, \infty})}\|t^{\frac{1}{2}}\nabla^2 u_j\|_{L^\infty_{T}(L^2)}+\|u\|_{L^{\infty}_{T}(\dot{H}^{\frac{1}{2}})}\|t^{\frac{1}{2}}\nabla^2 u_j\|_{L^{2}_{T}(L^{6})}\\
&\qquad+\|\nabla\,u\|_{L^2_{T}(\dot{B}^{\frac{1}{2}}_{2, \infty})}\|t^{\frac{1}{2}}\nabla^2 u_j\|_{L^{\infty}_{T}(L^{2})},
\end{align*}
where we used Proposition \ref{lorentz} in the last step. We thus obtain
\begin{equation}\label{j-est-basic-2-32-ddd}
\begin{split}
&\|t^{\frac{1}{2}}D_t u_j\|_{L^\infty_{T}(L^2)}+\|t^{\frac{1}{2}}\nabla
D_t u_j\|_{L^2_{T}(L^2)}\\
&\leq C(\|u_0\|_{\dot{H}^{\frac{1}{2}}})\, 2^{j}\|\dot{\Delta}_ju_0\|_{L^2}+C\|u_0\|_{\dot{H}^{\frac{1}{2}}}\|t^{\frac{1}{2}}\nabla^2 u_j\|_{L^{2}_{T}(L^{6})}.
\end{split}
\end{equation}
Whereas we deduce from \eqref{eqns-basic-2-3} and the classical theory on Stokes operator that
\begin{equation}\label{S3eq4}
\begin{split}
 &\|t^{\frac{1}{2}}(\na^2 u_j, \,\nabla \Pi_j)\|_{L^2_{T}(L^6)}\lesssim \|t^{\frac{1}{2}}\rho\,(
\partial_t u_j+u\cdot\nabla u_j)\|_{L^2_{T}(L^6)}\\
&\lesssim \|t^{\frac{1}{2}}
\partial_t u_j\|_{L^2_{T}(L^6)}+\| u\|_{L^{2}_{T}(L^{\infty})}\|t^{\frac{1}{2}}\nabla\,u_j\|_{L^{\infty}_{T}(L^{6})}\\
&\lesssim \|t^{\frac{1}{2}}
\partial_t \nabla{u}_j\|_{L^2_{T}(L^2)}+\| u\|_{L^{2}_{T}(L^{\infty})}\|t^{\frac{1}{2}}\nabla^2\,u_j\|_{L^{\infty}_{T}(L^{2})}.
\end{split}
\end{equation}
Notice that $\|u\|_{L^{\infty}}^2 \lesssim \|\nabla{u}\|_{L^2}\|\nabla^2\,u\|_{L^2}=\|t^{-\frac{1}{4}}\nabla{u}\|_{L^2}\|t^{\frac{1}{4}}\nabla^2\,u\|_{L^2},$ we deduce from \eqref{est-basic-2-18-a1} and \eqref{est-basic-2-30-conseq} that
\begin{equation}\label{grad-u-L4-L2}
\begin{split}
\| u\|_{L^{2}_{T}(L^{\infty})}^2 &\lesssim \int_0^T\|t^{-\frac{1}{4}}\nabla{u}\|_{L^2}\|t^{\frac{1}{4}}\nabla^2\,u\|_{L^2}\,dt\\
&\lesssim \| t^{-\frac{1}{4}}\nabla{u}\|_{L^{2}_{T}(L^{2})}\|t^{\frac{1}{4}}\nabla^2\,u\|_{L^{2}_{T}(L^{2})}\lesssim
\|u_0\|_{\dot H^{\frac{1}{2}}}.
\end{split}
\end{equation}
By substituting \eqref{j-est-basic-2-32} and \eqref{grad-u-L4-L2} into \eqref{S3eq4},
we find
\begin{equation}\label{S2eq6}
\begin{split}
\|t^{\frac{1}{2}}(\nabla^2 u_j, \,\nabla \Pi_j)\|_{L^2_T(L^6)}&\lesssim  \|t^{\frac{1}{2}}
\partial_t \nabla{u}_j\|_{L^2_{T}(L^2)}+\|u_0\|_{\dot H^{\frac{1}{2}}}\|t^{\frac{1}{2}}\nabla^2\,u_j\|_{L^{\infty}_{T}(L^{2})}
\\&
\lesssim 2^j\|\dot{\Delta}_ju_0\|_{L^2}.
\end{split}
\end{equation}
By inserting \eqref{S2eq6} into \eqref{j-est-basic-2-32-ddd}
we obtain  \eqref{est-basic-2-20eec}.

To prove \eqref{est-basic-2-20eed}, we first get, by using
\eqref{est-basic-2-5aa}, that
\begin{align*}
\|t^{\f12}\dot\D_j\na^2 u\|_{L^2_T(L^2)}\lesssim \sum_{j'\geq j}\|t^{\f12}\sqrt{\rho}\p_tu_{j'}\|_{L^2_T(L^2)}+2^{-j}\sum_{j'\leq j}\|t^{\f12}\na u_t\|_{L^2_T(L^2)}
\end{align*}
from which, \eqref{est-basic-2-20eea} and \eqref{j-est-basic-2-32}, we infer
\begin{align*}
\|t^{\f12}\dot\D_ju\|_{L^2_T(L^2)}\lesssim &\sum_{j'\geq j}\|\dot{\D}_{j'}u_0\|_{L^2}+2^{-j}\sum_{j'\leq j}2^{j'}\|\dot{\D}_{j'}u_0\|_{L^2}
\lesssim  c_j2^{-\f{j}2}\|u_0\|_{\dot H^{\f12}},
\end{align*}
which implies that
\beq \label{S3eq8}
\|t^{\frac{1}{2}}\nabla^2\,u\|_{L^2_T(\dot H^{\f12})}\lesssim \|u_0\|_{\dot H^{\f12}}.
\eeq
Combining \eqref{est-basic-2-30-conseq} with \eqref{S3eq8} and \eqref{grad-u-L4-L2} ensures \eqref{est-basic-2-20eed}.

If we assume in addition that $u_0\in \dot{B}^{\frac{1}{2}}_{2, r},$
we deduce from
 \eqref{est-basic-2-20aa} that
\begin{equation}\label{est-basic-2-20}
\begin{split}
&\|t^{\frac{1}{2}}\nabla\,u_j\|_{L^\infty_T(L^2)}+\|t^{\frac{1}{2}}(\partial_tu_j,\,\nabla^2\,u_j, \,\nabla\,\Pi_j)\|_{L^2_T(L^2)} \lesssim c_{j, r} 2^{-\frac{j}{2}}\|u_0\|_{\dot{B}^{\frac{1}{2}}_{2, r}},
\end{split}
\end{equation}
Along the same line to the derivation of \eqref{est-basic-2-15-aa},
 we get, by using \eqref{j-est-basic-2-32}, that
\begin{align*}
& \|t^{\frac{1}{2}}\nabla \dot{\Delta}_j u\|_{L_{T}^{\infty}(L^2)}+\|t^{\frac{1}{2}}\dot{\Delta}_j \partial_t u\|_{L^2_{T}(L^2)}\\
&=\sum_{j'\geq j}\bigl(\|t^{\frac{1}{2}}\nabla \dot{\Delta}_j u_{j'}\|_{L_{T}^{\infty}(L^2)}+\|t^{\frac{1}{2}}\dot{\Delta}_j \partial_t u_{j'}\|_{L^2_{T}(L^2)})\\
&\qquad +\sum_{j'\leq j-1}\bigl(\|t^{\frac{1}{2}}\nabla \dot{\Delta}_j u_{j'}\|_{L_{T}^{\infty}(L^2)}+\|t^{\frac{1}{2}}\dot{\Delta}_j \partial_t u_{j'}\|_{L^2_{T}(L^2)}\bigr)\\
&\lesssim \sum_{j'\geq j}\|\dot{\Delta}_ju_0\|_{L^2}+2^{-j}\sum_{j'\leq j-1}\|\nabla\dot{\Delta}_ju_0\|_{L^2},
\end{align*}
from which, we infer
\begin{equation*}\label{est-basic-total-1}
\begin{split}
& \|t^{\frac{1}{2}}\nabla \dot{\Delta}_j u\|_{L_{T}^{\infty}(L^2)}+\|t^{\frac{1}{2}}\dot{\Delta}_j \partial_t u\|_{L^2_{T}(L^2)}\\
&\lesssim \sum_{j'\geq j}c_{j', r}2^{-\frac{j'}{2}}\|u_0\|_{\dot{B}^{\frac{1}{2}}_{2, r}}+2^{-j}\sum_{j'\leq j-1}c_{j', r}2^{\frac{j'}{2}}\|u_0\|_{\dot{B}^{\frac{1}{2}}_{2, r}} \lesssim c_{j, r}2^{-\frac{j}{2}}\|u_0\|_{\dot{B}^{\frac{1}{2}}_{2, r}}.
\end{split}
\end{equation*}
This leads to \eqref{S3eq5}.
 This completes the proof of Proposition \ref{prop-time-energy-j-1}.
\end{proof}

\begin{prop}\label{prop-time-energy-total-1}
{\sl Under the assumption in Proposition \ref{prop-time-energy-j-1}, one has
\ben
&& \|\nabla\,u\|_{L^4_T(L^2)}+\|t^{\frac{1}{4}}\nabla\,u\|_{L^\infty_T(L^{2})}
\lesssim \|u_0\|_{\dot{B}^{\frac{1}{2}}_{2,\infty}},\label{est-basic-2-34aa}
\\&&
\|t^{\frac{3}{4}}(\partial_t u,\na^2u)\|_{L^\infty_T(L^2)}+\|t^{\frac{3}{4}}\nabla
\partial_t u\|_{L^2_T(L^2)}+\|t^{\frac{3}{4}}(\nabla^2 u, \,\nabla \Pi)\|_{L^2_T(L^6)}\lesssim\|u_0\|_{\dot{H}^{\frac{1}{2}}},\label{est-basic-2-34ab}
\\&&
\|t^{\frac{3}{4}}D_t u\|_{L^\infty_T(L^2)}+\|t^{\frac{3}{4}}\nabla
D_t u\|_{L^2_T(L^2)}
\lesssim \|u_0\|_{\dot{H}^{\frac{1}{2}}},\label{est-basic-2-34ac}
\\&&
\|t^{\frac{1}{2}}\nabla\,u\|_{L^2_T(L^{\infty})}
\lesssim \|u_0\|_{\dot{H}^{\frac{1}{2}}}.\label{est-basic-2-34ad}
\een
}
\end{prop}

\begin{proof}
Thanks to \eqref{est-basic-2-16} and \eqref{S3eq5}, we get 
 \begin{equation*} 
\begin{split}
&\|\nabla\,u\|_{L^4_T(L^2)}^2 \lesssim \|u\|_{{L}^\infty_T(\dot{B}^{\frac{1}{2}}_{2, \infty})}\|u\|_{\wt{L}^2_T(\dot{B}^{\frac{3}{2}}_{2, \infty})} \lesssim \|u_0\|_{\dot{B}^{\frac{1}{2}}_{2, \infty}}^2,\\
&\|t^{\frac{1}{4}}u\|_{L^{\infty}_T(\dot{H}^{1})}^2\lesssim\|u\|_{L^{\infty}_T(\dot{B}^{\frac{1}{2}}_{2, \infty})}
\|t^{\frac{1}{2}}u\|_{L^{\infty}_T(\dot{B}^{\frac{3}{2}}_{2, \infty})}\lesssim \|u_0\|_{\dot{B}^{\frac{1}{2}}_{2, \infty}}^2,
\end{split}
\end{equation*}
which follows \eqref{est-basic-2-34aa}.

Whereas by applying $\p_t$ to the momentum equation of \eqref{1.2} and then taking $L^2$ inner product of the resulting equation
with $\p_tu,$ we find
\begin{equation}\label{est-basic-total-111}
\begin{split}
&\frac{1}{2}\frac{d}{dt}\|\sqrt{\rho} \partial_t u\|_{L^2}^2+\|\nabla
\partial_t u\|_{L^2}^2  \\
&=-\int_{\R^3}\rho_t\partial_t u \cdot(\partial_t u +u\cdot\nabla u)\, dx-\int_{\R^3}\rho \partial_t u\cdot (u_t \cdot\nabla u)\, dx\eqdefa I+II+III.
\end{split}
\end{equation}
It follows from a similar derivation of \eqref{j-est-basic-2-22}-\eqref{j-est-basic-2-24} that
\begin{equation*}
\begin{split}
|I|+|II|&\lesssim \|u\|_{\dot{B}^{\frac{1}{2}}_{2, \infty}} \|\nabla\partial_t u\|_{L^{2}}^2+\|u\|_{\dot{B}^{\frac{1}{2}}_{2, \infty}}\|u\|_{\dot{B}^{\frac{3}{2}}_{2, \infty}} \|\nabla\partial_t u\|_{L^{2}}\|\nabla^2 u\|_{L^{2}},
\end{split}
\end{equation*}
and
\begin{equation*}
\begin{split}
|III|\lesssim
\|\sqrt{\rho}\|_{L^\infty}
\| \sqrt{\rho}\partial_t u\|_{L^2}
\|\partial_tu\|_{L^6} \|\nabla u\|_{L^3}
\lesssim
\| \sqrt{\rho}\partial_t u\|_{L^2}\|\nabla \partial_tu\|_{L^2} \|\nabla u\|_{\dot{H}^{\frac{1}{2}}}.
\end{split}
\end{equation*}
By substituting the above inequalities into  \eqref{est-basic-total-111} and using the smallness condition \eqref{smallness-u-1}, we find
\begin{equation}\label{est-basic-total-122}
\begin{split}
\frac{d}{dt}\|\sqrt{\rho} \,\partial_t u\|_{L^2}^2+\|\nabla
\partial_t u\|_{L^2}^2
\lesssim & \|u\|_{\dot{B}^{\frac{1}{2}}_{2, \infty}}^2\|u\|_{\dot{B}^{\frac{3}{2}}_{2, \infty}}^2\|\nabla^2 u\|_{L^{2}}^2+ \| \sqrt{\rho}\partial_t u\|_{L^2}^2 \|\nabla u\|_{\dot{H}^{\frac{1}{2}}}^2.
\end{split}
\end{equation}
While it follows from  \eqref{1.2} and \eqref{est-lem-prod-1} that
\begin{equation*}
\begin{split}
\|(\nabla^2 u, \grad\Pi)\|_{L^2}^2&\leq C( \|\sqrt{\rho}\,\partial_tu\|_{L^2}^2+\|u\cdot\nabla u\|_{L^2}^2) \leq C\bigl(\|\sqrt{\rho}\,\partial_tu\|_{L^2}^2+\|u\|_{\dot{B}^{\frac{1}{2}}_{2, \infty}}^2\|\nabla^2\,u\|_{L^2}^2\bigr).
\end{split}
\end{equation*}
Then under the assumption of \eqref{smallness-u-1}, we deduce from
\eqref{est-basic-2-16} that for any $t\leq T,$
\begin{equation}\label{est-basic-total-133}
\begin{split}
\|(\nabla^2 u(t), \,\grad\Pi(t))\|_{L^2}^2\lesssim \|\sqrt{\rho}\,\partial_tu(t)\|_{L^2}^2.
\end{split}
\end{equation}
By inserting \eqref{est-basic-total-133} into
 \eqref{est-basic-total-122}, we find
\begin{equation*}\label{est-basic-total-144}
\begin{split}
\frac{d}{dt}\|\sqrt{\rho} \,\partial_t u\|_{L^2}^2+\|\nabla
\partial_t u\|_{L^2}^2
\leq & C\bigl(\|u\|_{\dot{B}^{\frac{1}{2}}_{2, \infty}}^2\|u\|_{\dot{B}^{\frac{3}{2}}_{2, \infty}}^2+  \|\nabla u\|_{\dot{H}^{\frac{1}{2}}}^2\bigr)\|\sqrt{\rho} \,\partial_t u\|_{L^2}^2,
\end{split}
\end{equation*}
from which, we infer
\begin{equation*}
\begin{split}
&\frac{d}{dt}\|t^{\frac{3}{4}} \sqrt{\rho}\, \partial_t u\|_{L^2}^2+\|t^{\frac{3}{4}}\nabla
\partial_t u\|_{L^2}^2 \\
&\leq \f32\|t^{\frac{1}{4}}\partial_t u\|_{L^2}^2+C\bigl(\|u\|_{\dot{B}^{\frac{1}{2}}_{2, \infty}}^2\|u\|_{\dot{B}^{\frac{3}{2}}_{2, \infty}}^2
+\|\nabla u\|_{\dot{H}^{\frac{1}{2}}}^2\bigr)\|t^{\frac{3}{4}} \sqrt{\rho}\, \partial_t u\|_{L^{2}}^2.
\end{split}
\end{equation*}
Applying Gronwall's inequality gives rise to
\begin{equation*}\label{est-basic-total-155}
\begin{split}
&\|t^{\frac{3}{4}}\partial_t u\|_{L^\infty_t(L^2)}^2+\|t^{\frac{3}{4}}\nabla
\partial_t u\|_{L^2_t(L^2)}^2 \\
&\leq \f32 \|t^{\frac{1}{4}}u_t\|_{L^2_t(L^{2})}^2\,\exp\Bigl(C\bigl(\|u\|_{L^\infty_t(\dot{B}^{\frac{1}{2}}_{2, \infty})}^2\|u\|_{L^2_T(\dot{B}^{\frac{3}{2}}_{2, \infty})}^2+\| u\|_{L^2_t(\dot{H}^{\frac{3}{2}})}^2\bigr)\Bigr),
\end{split}
\end{equation*}
which together with \eqref{est-basic-2-18-a1}, \eqref{est-u-H12-infty-1}, \eqref{est-basic-2-30} and \eqref{est-basic-total-133}  ensures that
\begin{equation}\label{est-basic-total-166}
\begin{split}
&\|t^{\frac{3}{4}}(\partial_t u,\na^2u)\|_{L^\infty_T(L^2)}^2+\|t^{\frac{3}{4}}\nabla
\partial_t u\|_{L^2_T(L^2)}^2\leq C(\|u_0\|_{\dot{H}^{\frac{1}{2}}})\,\|u_0\|_{\dot{H}^{\frac{1}{2}}}^2\with\\
&\qquad C(\|u_0\|_{\dot{H}^{\frac{1}{2}}})\eqdefa C\bigl( 1+\|u_0\|_{\dot{H}^{\frac{1}{2}}}^2\bigr)\,\exp\Bigl(C\|u_0\|_{\dot{H}^{\frac{1}{2}}}^2 \Bigr).
\end{split}
\end{equation}

 Once again we deduce from \eqref{1.2} and classical theory on Stokes operator that
\begin{equation*}
\begin{split}
 \|t^{\frac{3}{4}}(\nabla^2 u, \,\nabla \Pi)\|_{L^2_T(L^6)}&\lesssim \|t^{\frac{3}{4}}\rho\,(
\partial_t u+u\cdot\nabla u)\|_{L^2_T(L^6)}\\
&\lesssim \|t^{\frac{3}{4}}
\partial_t u\|_{L^2_T(L^6)}+\|t^{\frac{1}{4}}u\|_{L^{\infty}_T(L^6)}\|t^{\frac{1}{2}}\nabla\,u\|_{L^{2}_T(L^{\infty})}\\
&\lesssim \|t^{\frac{3}{4}}\nabla
\partial_t u\|_{L^2_T(L^2)}+\|t^{\frac{1}{4}}\nabla{u}\|_{L^{\infty}_T(L^2)}\|t^{\frac{1}{4}}\nabla\,u\|_{L^{2}_T(L^6)}^{\frac{1}{2}}
\|t^{\frac{3}{4}}\nabla^2\,u\|_{L^2_T(L^6)}^{\frac{1}{2}},
\end{split}
\end{equation*}
from which, we infer
\begin{equation}\label{est-basic-total-177}
\begin{split}
 \|t^{\frac{3}{4}}(\nabla^2 u, \,\nabla \Pi)\|_{L^2_T(L^6)}\lesssim \|t^{\frac{3}{4}}
\nabla\partial_t u\|_{L^2_T(L^2)}+\|u\|_{L^{\infty}_T(\dot{H}^{\frac{1}{2}})}\|t^{\frac{1}{2}} u\|_{L^{\infty}_T(\dot{H}^{\frac{3}{2}})}\|t^{\frac{1}{4}}\nabla^2\,u\|_{L^{2}_T(L^2)}.
\end{split}
\end{equation}
Thanks to \eqref{est-basic-2-18-a1},  \eqref{est-basic-2-20eed},  \eqref{S3eq5}  and \eqref{est-basic-total-166},  we
conclude that
\begin{equation}\label{est-basic-total-188}
\begin{split}
 \|t^{\frac{3}{4}}(\nabla^2 u, \,\nabla \Pi)\|_{L^2_T(L^6)}
 \lesssim
 \|u_0\|_{\dot{H}^{\frac{1}{2}}},
\end{split}
\end{equation}
which together with \eqref{est-basic-total-166} ensures \eqref{est-basic-2-34ab}.

To prove \eqref{est-basic-2-34ac}, we get, by using Proposition \ref{lorentz}, that
\begin{equation*}\label{est-basic-total-211}
\begin{split}
&\|t^{\frac{3}{4}}D_t u\|_{L^\infty_T(L^2)}+\|t^{\frac{3}{4}}\nabla
D_t u\|_{L^2_T(L^2)}\\
&\lesssim \|t^{\frac{3}{4}}\partial_t u\|_{L^\infty_T(L^2)}+\|t^{\frac{3}{4}}u\cdot\nabla u\|_{L^\infty_T(L^2)}+\|t^{\frac{3}{4}}\nabla
\partial_t u\|_{L^2_T(L^2)}+\|t^{\frac{3}{4}}\nabla(u\cdot\nabla) u\|_{L^2_T(L^2)}\\
&\lesssim C(\|u_0\|_{\dot{H}^{\frac{1}{2}}})\,\|u_0\|_{\dot{H}^{\frac{1}{2}}}+\|u\|_{L^{\infty}_T(L^{3, \infty})}\|t^{\frac{3}{4}}\nabla u\|_{L^{\infty}_T(L^{6, 2})}+\|u\|_{L^{\infty}_T(L^{3})}\|t^{\frac{3}{4}}\nabla^2 u\|_{L^{2}_T(L^{6})}\\
&\qquad +\|\nabla\,u\|_{L^2_T(\dot{H}^{\frac{1}{2}})}\|t^{\frac{3}{4}}\nabla^2 u\|_{L^{\infty}_T(L^{2})}\\
&\lesssim C(\|u_0\|_{\dot{H}^{\frac{1}{2}}})\,\|u_0\|_{\dot{H}^{\frac{1}{2}}}+\|u\|_{L^{\infty}_T(\dot{B}_{2, \infty}^{\frac{1}{2}})}\|t^{\frac{3}{4}}\nabla^2 u\|_{L^{\infty}_T(L^{2})}+\|u\|_{L^{\infty}_T(\dot{H}^{\frac{1}{2}})}\|t^{\frac{3}{4}}\nabla^2 u\|_{L^{2}_T(L^{6})}\\
&\qquad +\|\nabla\,u\|_{L^2_T(\dot{H}^{\frac{1}{2}})}\|t^{\frac{3}{4}}\nabla^2 u\|_{L^{\infty}_T(L^{2})},
\end{split}
\end{equation*}
which along with  \eqref{est-basic-2-34ab} implies  \eqref{est-basic-2-34ac}.

On the other hand, we get, by using \eqref{est-basic-2-20eed}, \eqref{est-basic-total-188}, and  interpolating the inequality: $\|f\|_{L^\infty}\lesssim \|\na f\|_{L^2}^{\f1{2}}\|\na f\|_{L^{6}}^{\f1{2}},$ that
 \begin{equation*}
  \begin{split}
\|t^{\frac{1}{2}}\nabla\,u\|_{L^2_T(L^{\infty})}^2&
\lesssim
\int_{0}^Tt\,\|\nabla\,u(t)\|_{L^{\infty}}^2\,dt
\lesssim \int_{0}^T\,\|t^{\frac{1}{4}}\nabla^2\,u(t)\|_{L^{2}}
\|t^{\frac{3}{4}}\nabla^2\,u(t)\|_{L^{6}}\,dt
\\&
\lesssim
\|t^{\frac{1}{4}}\nabla^2\,u\|_{L^2_T(L^{2})} \|t^{\frac{3}{4}}\nabla^2\,u\|_{L^2_T(L^{6})}
\lesssim
\|u_0\|_{\dot{H}^{\frac{1}{2}}}^2,
  \end{split}
\end{equation*} which leads to  \eqref{est-basic-2-34ad}
This completes the proof of Proposition \ref{prop-time-energy-total-1}.
\end{proof}

Now we are in a position to complete the  proof of  Theorem \ref{thmmain-global}.

\begin{proof}[Proof of Theorem \ref{thmmain-global}]
By mollifying the initial data $(\rho_0, u_0)$ to be $(\rho_{0 \epsilon }, u_{0 \epsilon })$ , we deduce from the classical theory of inhomogeneous incompressible Navier-Stokes system (see \cite{A-G-Z-2} for instance) that \eqref{1.2} has a unique local solution $(\rho_{\epsilon}, u_{\epsilon})$ on $[0, T^{\ast}_{\epsilon})$. If $\mathfrak{c}$ is sufficiently small in \eqref{smallness-u-1}, we can show that \eqref{est-basic-2-16}, \eqref{est-basic-2-18-a1}, \eqref{est-basic-2-34ad}  hold for $(\rho_{\epsilon}, u_{\epsilon})$. Then a continuous argument shows  that $T_{\epsilon}^\ast=+\infty$.   Due to $u_0\in \dot B^{\f12}_{2,\infty}\cap\dot H^{\f12},$ we deduce from  \eqref{est-basic-2-18}
that for any $r\in [2,\infty)$
\beq \label{S4eq1}
 \|u_{\epsilon}\|_{\widetilde{L}^{\infty}(\R^+;\dot{B}^{\frac{1}{2}}_{2, r})}+\|u_{\epsilon}\|_{\widetilde{L}^2(\R^+;\dot{B}^{\frac{3}{2}}_{2, r})}
\leq C
\|u_0\|_{\dot{B}^{\frac{1}{2}}_{2, r}}.
\eeq
The derivation of \eqref{est-basic-2-30} ensures that
\begin{equation}\label{S4eq3}
\|{t}^{\frac{1}{4}}\p_tu_{\epsilon}\|_{L^2(\R^+;L^{2})}
\leq C\|u_0\|_{\dot{H}^{\frac{1}{2}}}.
\end{equation}
And it follows from the proof of \eqref{S3eq5} that
\beq\label{S4eq2}
\|t^{\frac{1}{2}} u_{\epsilon}\|_{\widetilde{L}^\infty(\R^+;\dot{B}^{\frac{3}{2}}_{2, r})}+\|t^{\frac{1}{2}}\partial_t u_{\epsilon}\|_{\widetilde{L}^2(\R^+;\dot{B}^{\frac{1}{2}}_{2, r})}
\lesssim \|u_0\|_{\dot{B}^{\frac{1}{2}}_{2, r}} \quad  \forall\,\, 2\leq r \leq +\infty.
\eeq
Furthermore, we deduce  from the proof of Proposition \ref{prop-time-energy-total-1} that
\begin{equation}\label{S4eq2}
\begin{split}
& \|\nabla\,u_{\epsilon}\|_{L^4(\R^+;L^2)}+\|t^{\frac{1}{4}}\nabla\,u_{\epsilon}\|_{L^\infty(\R^+;L^{2})}
\leq C\|u_0\|_{\dot{B}^{\frac{1}{2}}_{2,\infty}},
\\&
\|t^{\frac{3}{4}}(\partial_t u_{\epsilon}, \na^2u_{\epsilon},D_t^{\epsilon} u_{\epsilon})\|_{L^\infty(\R^+;L^2)}+\|t^{\frac{3}{4}}(\nabla
\partial_t u_{\epsilon},\nabla
D_t^{\epsilon} u_{\epsilon})\|_{L^2(\R^+;L^2)}\\
&+\|t^{\frac{3}{4}}(\nabla^2 u_{\epsilon}, \,\nabla \Pi_{\epsilon})\|_{L^2(\R^+;L^6)}
+\|t^{\frac{1}{2}}\nabla\,u_{\epsilon}\|_{L^2(\R^+;L^{\infty})}\leq C\|u_0\|_{\dot{H}^{\frac{1}{2}}},
\end{split}
\end{equation}
where $D_t^{\epsilon}\eqdefa \p_t+u_{\epsilon}\cdot\na.$

Then
exactly along the same line to the proof of Theorem 1.1 in \cite{HPZ-2013} (one may check page 644-645 of \cite{HPZ-2013}), we can complete the existence part of Theorem \ref{thmmain-global} by using  a standard compactness argument, which we omit details here. Moreover, we deduce \eqref{unif-est-1-1}
and \eqref{unif-est-1-2} from (\ref{S4eq1}-\ref{S4eq2}). The uniqueness part follows from Theorem \ref{thm-unique-HSWZ-1}.
Finally let's  prove  $u\in C([0,\infty); {\dot{H}^{\frac{1}{2}}}).$  Indeed it follows from  \eqref{unif-est-1-2} that $$\|u\|_{\widetilde{L}^{\infty}(\mathbb{R}^+; \dot{H}^{\frac{1}{2}})}\leq C\|u_0\|_{\dot{H}^{\frac{1}{2}}}.$$ Then for any $\varepsilon>0$, there is a positive $j_0=j_0(\varepsilon)\in \mathbb{N}$ so that
 \begin{equation*}\label{1.3aa}
\begin{split}
4\sum_{|j| \geq j_0}2^{j}\|\dot{\Delta}_ju\|_{L^{\infty}(\R^+; L^2)}^2<\varepsilon.
\end{split}
\end{equation*}
So that for any  $t \in [0, +\infty),\,h>0$, we have
\begin{equation*}\label{1.4-1}
\begin{split}
&\|u(t+h)-u(t)\|_{\dot{H}^{\frac{1}{2}}}^2=\sum_{j \in \mathbb{Z}}2^{ j}\|\dot{\Delta}_j(u(t+h)-u(t))\|_{L^2}^2\\
&\leq \sum_{|j| \leq j_0-1}2^{j}\|\dot{\Delta}_j(u(t+h)-u(t))\|_{L^2}^2+4\sum_{|j| \geq j_0}2^{j}\|\dot{\Delta}_ju\|_{L^{\infty}_T(L^2)}^2\\
&\leq 2^{j_0} \|u(t+h)-u(t)\|_{L^2}^2+\varepsilon,
\end{split}
\end{equation*}
from which, we infer
\begin{equation}\label{1.4-2}
\begin{split}
\|u(t+h)-u(t)\|_{\dot{H}^{\frac{1}{2}}}^2&\leq 2^{j_0} \|\int_t^{t+h}\tau^{-\frac{1}{4}}\, \tau^{\frac{1}{4}}\partial_{\tau}u(\tau)\,d\tau\|_{L^2}^2+\varepsilon\\
&\leq 2^{j_0}  \|\tau^{-\frac{1}{4}}\|_{L^2([t, t+h])}^2\|\tau^{\frac{1}{4}}\pa_{\tau} u\|_{L^2([t, t+h])}\|_{L^2}^2 +\varepsilon\\
&\leq  C\, 2^{j_0+2}\|u_0\|_{\dot{H}^{\frac{1}{2}}}\,h^{\frac{1}{2}} +\varepsilon,
\end{split}
\end{equation}
where we used the fact $\|t^{\frac{1}{4}}\pa_t u\|_{L^2(\R^+; L^2)} \leq C \|u_0\|_{\dot{H}^{\frac{1}{2}}}$ of \eqref{unif-est-1-2}.
This shows that $u\in C([0,\infty);$ $ \dot{H}^{\frac{1}{2}})$, and we completes the proof of Theorem \ref{thmmain-global}.
\end{proof}

\renewcommand{\theequation}{\thesection.\arabic{equation}}
\setcounter{equation}{0}

\section{The proof of Theorem \ref{thmmain-regularity}}\label{Sect4}

The goal of this section is to present the  proof of Theorem \ref{thmmain-regularity}.
With the {\it a priori} estimates derived in the previous section, here we shall first derive the globally-in-time Lipschitz estimate
\eqref{col-pressure-Lip-1} below for the velocity field if we assume in addition that $u_0\in \dot B^{\f12}_{2,1}.$  Toward this, we first derive the following
estimate via energy method:

\begin{prop}\label{prop-time-int-L1-1}
{\sl Under the assumptions of Proposition \ref{prop-time-energy-j-1}, we have
\begin{equation}\label{est-basic-2-54-ccc}
\begin{split}
&\|t \,D_t u_j \|_{L^\infty_T(L^2)}+\|t\,\nabla
D_t u_j\|_{L^2_T(L^2)}\lesssim\|\dot\Delta_ju_0\|_{L^2}.
\end{split}
\end{equation}
If in addition $u_0\in \dot{B}^{\frac{1}{2}}_{2, 1}$, we have
\begin{equation}\label{est-basic-2-66}
\begin{split}
&\|t^{\frac{1}{2}}(\partial_tu,\,\nabla^2\,u,\,\nabla\Pi)\|_{L^{4, 1}_T(L^2)}\lesssim\|u_0\|_{\dot{B}^{\frac{1}{2}}_{2, 1}},
\end{split}
\end{equation}
and
\begin{equation}\label{est-basic-2-104}
\begin{split}
&\|t \,(\nabla^2 u, \,\pa_t u,\,\nabla\Pi)\|_{L^{4, 1}_T(L^6)}  \lesssim  \|u_0\|_{\dot{B}^{\frac{1}{2}}_{2, 1}},
\end{split}
\end{equation} where the Lorentz norm of $L^{4,1}_T$ is defined by Definition \ref{espace_lorentz}.}
\end{prop}

\begin{proof} We first get, by multiplying \eqref{j-est-basic-2-26} by $t^2,$ that
\begin{equation*}
\begin{split}
&\frac{d}{dt}\|t\,\sqrt{\rho} \,\partial_t u_j\|_{L^2}^2+\|t\,\nabla
\partial_t u_j\|_{L^2}^2 \leq 2 \|t^{\frac{1}{2}}\sqrt{\rho} \,\partial_t u_j\|_{L^2}^2\\
&\qquad+C\Bigl(\|u\|_{\dot{B}^{\frac{1}{2}}_{2, \infty}}^2\|u\|_{\dot{B}^{\frac{3}{2}}_{2, \infty}}^2\|t\,\sqrt{\rho} \,\partial_t u_j\|_{L^{2}}^2+\|t^{\frac{1}{4}}\,u_t\|_{L^{2}}^2\|t^{\frac{1}{2}}\nabla u_j\|_{L^{2}} \|t\,\partial_t u_j\|_{L^{2}}\Bigr),
\end{split}
\end{equation*}
which together with \eqref{bdd-density-res-1} implies
\begin{equation*}
\begin{split}
&\frac{d}{dt}\|t\,\sqrt{\rho} \,\partial_t u_j\|_{L^2}^2+\|t\,\nabla
\partial_t u_j\|_{L^2}^2 \leq  2\|t^{\frac{1}{2}}\sqrt{\rho} \,\partial_t u_j\|_{L^2}^2\\
&\qquad+C\Bigl(\bigl(\|u\|_{\dot{B}^{\frac{1}{2}}_{2, \infty}}^2\|u\|_{\dot{B}^{\frac{3}{2}}_{2, \infty}}^2+\|t^{\frac{1}{4}}\,u_t\|_{L^{2}}^2\bigr)\|t\,\sqrt{\rho} \,\partial_t u_j\|_{L^{2}}^2+\|t^{\frac{1}{4}}\,u_t\|_{L^{2}}^2\|t^{\frac{1}{2}}\nabla u_j\|_{L^{2}}^2\Bigr).
\end{split}
\end{equation*}
Applying Gronwall's inequality gives rise to
\begin{equation*}\label{est-basic-2-53}
\begin{split}
&\|t\, \partial_t u_j\|_{L^\infty_t(L^2)}^2+\|t\,\nabla
\partial_t u_j\|_{L^2_t(L^2)}^2 \\
&\leq C\bigl(\|t^{\frac{1}{2}}\partial_t u_j\|_{L^2_t(L^{2})}^2+\|{t}^{\frac{1}{4}}\,u_{t}\|_{L^2_t(L^{2})}^2\|t^{\frac{1}{2}}\nabla u_j\|_{L^\infty_T(L^{2})}^2\bigr) \\
&\qquad \quad\times\exp\Bigl(C\bigl(\|u\|_{L^\infty_t(\dot{B}^{\frac{1}{2}}_{2, \infty})}^2\|u\|_{L^2_t(\dot{B}^{\frac{3}{2}}_{2, \infty})}^2+\|{t}^{\frac{1}{4}}\,u_{t}\|_{L^2_t(L^{2})}^2\bigr)\Bigr)\quad\forall\ t\in [0,T].
\end{split}
\end{equation*}
By inserting \eqref{est-basic-2-16}, \eqref{est-basic-2-20aa} and \eqref{est-basic-2-30}   into the above inequality and using \eqref{est-basic-2-5aa}, we obtain
\begin{equation}\label{est-basic-2-54}
\begin{split}
&\|t \,(\partial_t u_j, \,\nabla^2 u_j)\|_{L^\infty_T(L^2)}+\|t\,\nabla
\partial_t u_j\|_{L^2_T(L^2)}\leq C(\|u_0\|_{\dot{H}^{\frac{1}{2}}})\,\|\dot\Delta_ju_0\|_{L^2}.
\end{split}
\end{equation}

While we deduce from \eqref{eqns-basic-2-3} and the classical theory on Stokes operator that
\begin{equation*}
\begin{split}
\|t \,(\nabla^2 u_j, \,\nabla\Pi_j)\|_{L^6}&\lesssim \|t \,\partial_t u_j\|_{L^6}+\|t\,u\cdot \nabla\,u_j\|_{L^6}\\
&\lesssim \|t \,\partial_t u_j\|_{L^6}+t\,\|u\|_{L^6}\|\nabla\,u_j\|_{L^\infty}\\
&\lesssim \|t \,\partial_t u_j\|_{L^6}+ \|\nabla{u}\|_{L^2}\|t\,\nabla\,u_j\|_{L^6}^{\frac{1}{2}}\|t\,\nabla^2\,u_j\|_{L^6}^{\frac{1}{2}}.
\end{split}
\end{equation*}
Applying Young's inequality yields
\begin{equation}\label{est-basic-2-72}
\|t \,(\nabla^2 u_j, \,\nabla\Pi_j)\|_{L^6}\lesssim \|t \,\partial_t u_j\|_{L^6}+ \|\nabla{u}\|_{L^2}^2\|t\,\nabla^2\,u_j\|_{L^2}.
\end{equation}
In particular, we have
\begin{equation*}
\|t \,(\nabla^2 u_j, \,\nabla\Pi_j)\|_{L^2_T(L^6)}\lesssim \|t \,\nabla\partial_t u_j\|_{L^2_T(L^2)}+ \|\nabla{u}\|_{L^4_T(L^2)}^2\|t\,\nabla^2\,u_j\|_{L^\infty_T(L^2)},
\end{equation*}
which along with \eqref{grad-u-L4-L2} and \eqref{est-basic-2-54} ensures that
\begin{equation}\label{est-basic-2-75}
\|t \,(\nabla^2 u_j, \,\nabla\Pi_j)\|_{L^2_T(L^6)}\leq C(\|u_0\|_{\dot{H}^{\frac{1}{2}}})\bigl(1+\|u_0\|_{\dot{B}^{\frac{1}{2}}_{2, \infty}}^2\bigr)\|\dot\Delta_ju_0\|_{L^2}.
\end{equation}

On the other hand, we get, by applying Propositions \ref{Neil} and \ref{lorentz},  that
\begin{equation*}
\begin{split}
&\|t \,D_t  u_j \|_{L^\infty_T(L^2)}+\|t\,\nabla
D_t u_j\|_{L^2_T(L^2)}\\
 &\lesssim \|t  \partial_t u_j\|_{L^\infty_T(L^2)}+\|u \|_{L^\infty_T(L^{3, \infty})}\|t \nabla u_j \|_{L^\infty_T(L^{6, 2})}+\|t\,\nabla
\partial_t u_j\|_{L^2_T(L^2)}\\
 &\quad+\|t^{\frac{1}{2}}\nabla\,
u \|_{L^\infty_T(L^{3, \infty})}\|t^{\frac{1}{2}}\nabla u_j \|_{L^2_T(L^{6, 2})}+\|u\|_{L^{\infty}_T(L^{3})}\|t \nabla^2 u_j\|_{L^{2}_T(L^6)}\\
 &\lesssim \|t  \partial_t u_j\|_{L^\infty_T(L^2)}+\|u \|_{L^\infty_T(\dot{B}^{\frac{1}{2}}_{2, \infty})}\|t \nabla^2 u_j \|_{L^\infty_T(L^{2})}+\|t\,\nabla
\partial_t u_j\|_{L^2_T(L^2)}\\
 &\quad+\|t^{\frac{1}{2}}\nabla\,
u \|_{L^\infty_T(\dot{B}^{\frac{1}{2}}_{2, \infty})}\|t^{\frac{1}{2}}\nabla^2 u_j \|_{L^2_T(L^{2})}+\|u\|_{L^{\infty}_T(\dot{H}^{\frac{1}{2}})}\|t \nabla^2 u_j\|_{L^{2}_T(L^6)},
\end{split}
\end{equation*}
which along with \eqref{est-basic-2-16}, \eqref{j-est-basic-2-32}, \eqref{est-basic-2-54} and \eqref{est-basic-2-75} ensures \eqref{est-basic-2-54-ccc}.

Whereas it follows from \eqref{est-basic-2-20eea} and \eqref{j-est-basic-2-32} that
\begin{align*}\label{est-basic-2-61}
&\|t^{\frac{1}{2}}(\partial_tu_j,\,\nabla^2\,u_j)\|_{L^2_T(L^2)} \lesssim d_{j}2^{-\frac{j}{2}}\|u_0\|_{\dot{B}^{\frac{1}{2}}_{2, 1}},\\
&\|t^{\frac{1}{2}}(\partial_tu_j,\,\nabla^2\,u_j)\|_{L^\infty_T(L^2)} \lesssim d_{j}2^{\frac{j}{2}}\|u_0\|_{\dot{B}^{\frac{1}{2}}_{2, 1}},
\end{align*}
from which and the interpolation theorem: $L^{4, 1}_T(L^2)=\bigl(L^2_T(L^2), L^\infty_T(L^2)\bigr)_{\frac{1}{2}, 1},$ we infer
\begin{equation*}\label{est-basic-2-62}
\begin{split}
&\|t^{\frac{1}{2}}(\partial_tu_j,\,\nabla^2\,u_j)\|_{L^{4, 1}_T(L^2)}\lesssim  d_{j}\|u_0\|_{\dot{B}^{\frac{1}{2}}_{2, 1}}.
\end{split}
\end{equation*}
Then we deduce that
\begin{equation}\label{est-basic-2-63}
\begin{split}
 \|t^{\frac{1}{2}}(\partial_tu,\,\nabla^2\,u)\|_{L^{4, 1}_T(L^2)}&=\|\sum_{j\in \mathbb{Z}}t^{\frac{1}{2}}(\partial_tu_j,\,\nabla^2\,u_j)\|_{L^{4, 1}_T(L^2)}\\
&\lesssim \sum_{j\in \mathbb{Z}}\|t^{\frac{1}{2}}(\partial_tu_j,\,\nabla^2\,u_j)\|_{L^{4, 1}_T(L^2)}\lesssim\|u_0\|_{\dot{B}^{\frac{1}{2}}_{2, 1}}.
\end{split}
\end{equation}

In view of \eqref{est-lem-prod-1},  we get, by applying Propositions \ref{Neil} and \ref{lorentz},  that
\begin{equation*}
\begin{split}
 \|t^{\frac{1}{2}}\nabla\Pi\|_{L^{4, 1}_T(L^2)}&=\|t^{\frac{1}{2}}(\Delta\,u-\rho(\partial_tu+u\cdot\nabla\,u))\|_{L^{4, 1}_T(L^2)}\\
&\lesssim\|t^{\frac{1}{2}}(\partial_tu,\,\nabla^2\,u)\|_{L^{4, 1}_T(L^2)}+\| u\|_{L^{\infty}_T(L^{3, \infty})}\|t^{\frac{1}{2}} \nabla\,u\|_{L^{4, 1}_T(L^{6, 2})}\\
&\lesssim\|t^{\frac{1}{2}}(\partial_tu,\,\nabla^2\,u)\|_{L^{4, 1}_T(L^2)}+\|u\|_{L^{\infty}_T(\dot{B}^{\frac{1}{2}}_{2, \infty})}\|t^{\frac{1}{2}}\nabla^2\,u\|_{L^{4, 1}_T(L^2)},
\end{split}
\end{equation*}
which together with \eqref{est-basic-2-16} and \eqref{est-basic-2-63} implies
\begin{equation*}
\begin{split}
&\|t^{\frac{1}{2}}\nabla\Pi\|_{L^{4, 1}_T(L^2)}\lesssim\|u_0\|_{\dot{B}^{\frac{1}{2}}_{2, 1}}.
\end{split}
\end{equation*}
This together with \eqref{est-basic-2-63} ensures \eqref{est-basic-2-66}.

Finally let us turn to the proof of \eqref{est-basic-2-104}. Indeed
again thanks to \eqref{est-basic-2-72}, one has
\begin{equation*}
\begin{split}
&\|t \,(\nabla^2 u_j, \,\nabla\Pi_j)\|_{L^\infty_T(L^6)}\lesssim \|t\partial_t u_j\|_{L^\infty_T(L^6)}+ \|t^{\frac{1}{4}}\nabla{u}\|_{L^\infty_T(L^2)}^2\|t^{\frac{1}{2}}\nabla^2\,u_j\|_{L^\infty_T(L^2)},
\end{split}
\end{equation*}
which together with \eqref{j-est-basic-2-32} and \eqref{est-basic-2-34aa} implies
\begin{equation}\label{est-basic-2-78}
\begin{split}
\|t \,(\nabla^2 u_j, \,\nabla\Pi_j)\|_{L^\infty_T(L^6)}
&\lesssim \|t \,\partial_t u_j\|_{L^\infty_T(L^6)}+ \|t^{\frac{1}{2}}\nabla^2\,u_j\|_{L^\infty_T(L^2)} \\
&\lesssim \|t \,\partial_t u_j\|_{L^\infty_T(L^6)}+ 2^{j}\|\Delta_ju_0\|_{L^2}.
\end{split}
\end{equation}
To complete the proof of \eqref{est-basic-2-104}, we need the following lemma, the proof of which will be postponed after we
finish the proposition.
\begin{lem}\label{S3lem1}
{\sl
If  $u_0\in \dot{B}^{\frac{1}{2}}_{2, 1}$, one has
\begin{equation}\label{est-basic-2-99}
\|t\,\partial_t u_j\|_{L^{\infty}_T(L^6)} \lesssim 2^{j}\|\Delta_ju_0\|_{L^2}.
\end{equation}}
\end{lem}

By combining \eqref{est-basic-2-78} with \eqref{est-basic-2-99}, we obtain
\begin{equation*}
\begin{split}
&\|t \,(\nabla^2 u_j, \,\partial_t u_j, \,\nabla\Pi_j)\|_{L^\infty_T(L^6)}  \lesssim  2^{j}\|\Delta_ju_0\|_{L^2},
\end{split}
\end{equation*}
from which, \eqref{est-basic-2-75} and   the interpolation theorem:
$L^{4, 1}_T(L^2)=\bigl(L^2_T(L^2), L^\infty_T(L^2)\bigr)_{\frac{1}{2}, 1},$ we infer
\begin{equation}\label{S5eq1}
\|t \,(\nabla^2 u_j, \,\nabla\Pi_j)\|_{L^{4, 1}_T(L^6)}  \lesssim  2^{\frac{j}{2}}\|\Delta_ju_0\|_{L^2}.
\end{equation}
Due to the momentum equations in \eqref{model-3d-freq-1}, we find
\begin{equation}\label{est-t-uj-L6-1}
\begin{split}
\|t \,\pa_t u_j\|_{L^{4, 1}_T(L^6)} & \lesssim \|t \,(\nabla^2 u_j, \, \nabla\Pi_j)\|_{L^{4, 1}_T(L^6)} + \|t \,u\cdot\nabla u_j\|_{L^{4, 1}_T(L^6)}\\
&\lesssim   \|t \,(\nabla^2 u_j, \, \nabla\Pi_j)\|_{L^{4, 1}_T(L^6)}+ \|{t}^{\frac{1}{2}}  u\|_{L^{\infty}_T(L^{\infty})}\|{t}^{\frac{1}{2}}  \nabla u_j\|_{L^{4, 1}_T(L^6)}.
\end{split}
\end{equation}
Whereas we get, by using the interpolation inequality: $\|a\|_{L^\infty}\lesssim \|a\|_{\dot B^{\f12}_{2,\infty}}^{\f13}\|\na^2a\|_{L^2}^{\f23},$
\eqref{est-basic-2-18} and \eqref{est-basic-2-34ab}, that
\beno
\|{t}^{\frac{1}{2}} u\|_{L^{\infty}_T(L^{\infty})}\lesssim   \|u\|_{L^\infty_T(\dot B^{\f12}_{2,\infty})}^{\f13}
\|t^{\f34}\na^2u\|_{L^\infty_T(L^2)}^{\f23}\lesssim \|u_0\|_{\dot H^{\f12}},
\eeno
from which, \eqref{S5eq1} and \eqref{est-t-uj-L6-1}, we infer
\begin{align*}
\|t \,\pa_t u_j\|_{L^{4, 1}_T(L^6)} \lesssim  \|t \,(\nabla^2 u_j, \, \nabla\Pi_j)\|_{L^{4, 1}_T(L^6)}+ \|u_0\|_{\dot{B}^{\frac{1}{2}}_{2, 1}}\|{t}^{\frac{1}{2}} \nabla^2 u_j\|_{L^{4, 1}_T(L^2)}\lesssim 2^{\frac{j}{2}}\|\Delta_ju_0\|_{L^2}.
\end{align*}
We thus deduce that
\begin{equation*}\label{est-basic-2-103}
\|t \,(\nabla^2 u_j, \,\pa_t u_j,\,\nabla\Pi_j)\|_{L^{4, 1}_T(L^6)}  \lesssim  2^{\frac{j}{2}}\|\Delta_ju_0\|_{L^2}\lesssim d_j \|u_0\|_{\dot{B}^{\frac{1}{2}}_{2, 1}},
\end{equation*}
from which, we arrive at
\begin{equation*}
\|t \,(\nabla^2 u, \,\pa_t u,\,\nabla\Pi)\|_{L^{4, 1}_T(L^6)}
 \lesssim \sum_{j\in \mathbb{Z}}\|t \,(\nabla^2 u_j, \,\pa_t u_j,\,\nabla\Pi_j)\|_{L^{4, 1}_T(L^6)}\lesssim \|u_0\|_{\dot{B}^{\frac{1}{2}}_{2, 1}},
\end{equation*}
and \eqref{est-basic-2-104} follows. This completes the proof of Proposition \ref{prop-time-int-L1-1}.
\end{proof}

Let us complete the proof of Lemma \ref{S3lem1}.

\begin{proof}[Proof of Lemma \ref{S3lem1}]
We get, by first applying $\partial_t$ to the $u_j$ equation of \eqref{model-3d-freq-1} and then taking the $L^2$
inner product of the resulting equation with $|\pa_t u_j|^4\pa_t u_j,$ that
\begin{equation}\label{est-basic-2-91}
\begin{split}
&\frac{1}{6}\frac{d}{dt}\|{\rho}^{\frac{1}{6}} \,\partial_t u_j\|_{L^6}^6+\|\nabla
|\partial_t u_j|^3\|_{L^2}^2 \\
&=-\int_{\R^3}\rho_t|\pa_t u_j|^4\partial_t u_j \cdot(\partial_t u_j +u\cdot\nabla u_j)\, dx-\int_{\R^3}\rho |\pa_t u_j|^4\partial_t u_j \cdot (u_t \cdot\nabla u_j)\, dx\\
&\eqdefa I+II+III.
\end{split}
\end{equation}
Next we handle term by term above.

By using the transport equation of \eqref{1.2} and integrating by parts, we have
\begin{equation*}
\begin{split}
&I=-\int_{\R^3}\rho_t|\partial_t u_j|^6\,dx=\int_{\R^3}\nabla\cdot(\rho\,u)|\partial_t u_j|^6\,dx=-2\int_{\R^3}\rho\, |\partial_t u_j|^3 \,(u\cdot \nabla)|\partial_t u_j|^3\,dx,
\end{split}
\end{equation*}
from which and the embedding: $\dot{B}^{\frac{1}{2}}_{2, \infty}(\mathbb{R}^3) \hookrightarrow L^{3, \infty}(\mathbb{R}^3)$ and $\dot{H}^{1}(\mathbb{R}^3) \hookrightarrow L^{6, 2}(\mathbb{R}^3)$ (see Proposition \ref{lorentz}), we infer
\begin{equation*}
\begin{split}
&I\lesssim \|\rho\|_{L^\infty}\|u\|_{L^{3, \infty}}\|\nabla|\partial_t u_j|^3\|_{L^2} \||\partial_t u_j|^3\|_{L^{6, 2}}\lesssim \|u\|_{\dot{B}^{\frac{1}{2}}_{2, \infty}}\|\nabla|\partial_t u_j|^3\|_{L^2}^2.
\end{split}
\end{equation*}

Similarly for $II$, we have
\begin{equation*}
\begin{split}
II&=-\int_{\R^3}\rho_t |\pa_t u_j|^4\partial_t u_j  \cdot (u\cdot\nabla u_j)\, dx
=\int_{\R^3}\nabla\cdot(\rho\,u)\bigl[ |\pa_t u_j|^4\partial_t u_j  \cdot (u\cdot\nabla u_j)\bigr]\,dx\\
&=-\int_{\R^3} \rho\,u \cdot\bigl[\nabla( |\pa_t u_j|^4\partial_t u_j ) \cdot (u\cdot\nabla u_j)\bigr]\,dx-\int_{\R^3} \rho\,u\cdot\bigl[ |\pa_t u_j|^4\partial_t u_j  \cdot (\nabla\,u\cdot\nabla u_j)\bigr]\,dx\\
&\qquad-\int_{\R^3} \rho\,u\cdot\bigl[ |\pa_t u_j|^4\partial_t u_j  \cdot (u\cdot\nabla^2 u_j)\bigr]\,dx\eqdefa \sum_{k=1}^3II_{k}.
\end{split}
\end{equation*}
By applying H\"older inequality and Proposition \ref{lorentz}, we find
\begin{align*}
|II_1|&\lesssim\|\rho\|_{L^\infty}\|u\|_{L^{\infty}}^2 \|\nabla|\partial_t u_j|^3\|_{L^{2}} \||\pa_t u_j|^2\|_{L^{3}}\|\nabla u_j\|_{L^{6}}\\
&\lesssim \|u\|_{L^{\infty}}^2 \|\nabla|\partial_t u_j|^3\|_{L^{2}} \|\pa_t u_j\|_{L^{6}}^2\|\nabla^2 u_j\|_{L^{2}},\\
|II_2|&\lesssim \|\rho\|_{L^\infty}\|u\|_{L^{\infty}}\|\nabla\,u\|_{L^{3, \infty}}\||\pa_t u_j|^2\|_{L^{3}} \||\partial_t u_j|^3\|_{L^{6, 2}} \|\nabla u_j\|_{L^{6, 2}}\\
&\lesssim  \|u\|_{L^{\infty}}\|\nabla\,u\|_{L^{3, \infty}}\|\nabla|\partial_t u_j|^3\|_{L^{2}} \|\pa_t u_j\|_{L^{6}}^2\|\nabla^2 u_j\|_{L^{2}},\\
|II_3|&\lesssim \|\rho\|_{L^\infty}  \|u\|_{L^{\infty}}^2 \||\pa_t u_j|^2\|_{L^{3}}\||\partial_t u_j|^3\|_{L^{6, 2}} \|\nabla^2 u_j\|_{L^{2}}\\
&\lesssim \|u\|_{L^{\infty}}^2 \|\nabla|\partial_t u_j|^3\|_{L^{2}} \|\pa_t u_j\|_{L^{6}}^2\|\nabla^2 u_j\|_{L^{2}}.
\end{align*}
Hence, 
we obtain
\begin{equation*}
|II|\lesssim \bigl(\|u\|_{L^{\infty}}^2+\|u\|_{L^{\infty}}\|\nabla\,u\|_{L^{3, \infty}}\bigr) \|\nabla|\partial_t u_j|^3\|_{L^{2}} \|\pa_t u_j\|_{L^{6}}^2\|\nabla^2 u_j\|_{L^{2}}.
\end{equation*}

Finally it 
is easy to observe  that
\begin{equation*}
\begin{split}
|III|&\lesssim \|\rho\|_{L^\infty}\|u\|_{L^{\infty}}^2 \||\pa_t u_j|^3\|_{L^{6}} \||\pa_t u_j|^2\|_{L^{3}} \|\nabla^2 u_j\|_{L^{2}}\\
&\lesssim \|u\|_{L^{\infty}}^2 \|\nabla|\partial_t u_j|^3\|_{L^{2}} \|\pa_t u_j\|_{L^{6}}^2\|\nabla^2 u_j\|_{L^{2}}.
\end{split}
\end{equation*}

By substituting the above estimates into
\eqref{est-basic-2-91}, we arrive at
\begin{equation*}
\begin{split}
&\frac{1}{6}\frac{d}{dt}\|{\rho}^{\frac{1}{6}} \,\partial_t u_j\|_{L^6}^6+\|\nabla
|\partial_t u_j|^3\|_{L^2}^2 \\
&\leq C\bigl(\|u\|_{L^{\infty}}^2+\|u\|_{L^{\infty}}\|\nabla\,u\|_{L^{3, \infty}}\bigr) \|\nabla|\partial_t u_j|^3\|_{L^{2}} \|\pa_t u_j\|_{L^{6}}^2\|\nabla^2 u_j\|_{L^{2}},
\end{split}
\end{equation*}
applying Young's inequality yields
\begin{equation*}
\frac{d}{dt}\|{\rho}^{\frac{1}{6}} \,\partial_t u_j\|_{L^6}^6+\|\nabla
|\partial_t u_j|^3\|_{L^2}^2 \leq C(\|u\|_{L^{\infty}}^4+\|u\|_{L^{\infty}}^2\|\nabla\,u\|_{L^{3, \infty}}^2)\|\pa_t u_j\|_{L^{6}}^4\|\nabla^2 u_j\|_{L^{2}}^2.
\end{equation*}
Then we get, by multiplying $t^6$ to the above inequality, that
\begin{equation*}
\begin{split}
 \frac{d}{dt}\|t\,{\rho}^{\frac{1}{6}} \,\partial_t u_j\|_{L^6}^6+&\|t^3\,\nabla
|\partial_t u_j|^3\|_{L^2}^2 \leq 6\|t^{\frac{1}{2}}{\rho}^{\frac{1}{6}} \,\partial_t u_j\|_{L^6}^2 \|t\,{\rho}^{\frac{1}{6}} \,\partial_t u_j\|_{L^6}^4\\
&+C\bigl(\|t^{\frac{1}{2}}u\|_{L^{\infty}}^4+\|t^{\frac{1}{2}}u\|_{L^{\infty}}^2\|t^{\frac{1}{2}}\nabla\,u\|_{L^{3, \infty}}^2\bigr)\|t\,{\rho}^{\frac{1}{6}} \,\pa_t u_j\|_{L^{6}}^4\|\nabla^2 u_j\|_{L^{2}}^2,
\end{split}
\end{equation*}
from which, we deduce that
\begin{equation}\label{est-basic-2-98}
\begin{split}
\|t\,\partial_t u_j&\|_{L^{\infty}_T(L^6)}^2\leq C \|t^{\frac{1}{2}}\partial_t u_j\|_{L^2_T(L^6)}^2\\
&+C\|t^{\frac{1}{2}}u\|_{L^{\infty}_T(L^{\infty})}^2(\|t^{\frac{1}{2}}u\|_{L^{\infty}_T(L^{\infty})}^2+\|t^{\frac{1}{2}}\nabla\,u\|_{L^{\infty}_T(L^{3, \infty})}^2)\|\nabla^2 u_j\|_{L^{2}_T(L^{2})}^2.
\end{split}
\end{equation}
Yet it follows from  \eqref{j-est-basic-2-32}, \eqref{S3eq5} and \eqref{est-basic-2-10} that
\begin{equation*}
\begin{split}
&\|t^{\frac{1}{2}}\partial_t u_j\|_{L^2_T(L^6)} \lesssim \|t^{\frac{1}{2}}\nabla\partial_t u_j\|_{L^2_T(L^2)}\lesssim 2^{j}\|\Delta_ju_0\|_{L^2},\\
&\|t^{\frac{1}{2}}u\|_{L^{\infty}_T(L^{\infty})}+\|t^{\frac{1}{2}}\nabla\,u\|_{L^{\infty}_T(L^{3, \infty})} \lesssim \|t^{\frac{1}{2}}u\|_{L^{\infty}_T(\dot{B}^{\frac{3}{2}}_{2, 1})} \lesssim \|u_0\|_{\dot{B}^{\frac{1}{2}}_{2, 1}},\\
&\|\nabla^2 u_j\|_{L^2_T(L^2)})\lesssim 2^{j}\|\Delta_ju_0\|_{L^2}.
\end{split}
\end{equation*}
By inserting the above estimates into \eqref{est-basic-2-98}, we obtain \eqref{est-basic-2-99}. This completes the proof of Lemma \ref{S3lem1}.
\end{proof}

\begin{col}\label{S4col1}
{\sl Let $u_0\in \dot{B}^{\frac{1}{2}}_{2, 1}.$ Then under the assumptions in Proposition \ref{prop-basicE-1},
we have
\begin{equation}\label{pressure-u-11}
\begin{split}
&\|(\Delta\,u,\,\nabla\,\Pi, \, \partial_t u)\|_{L^1_T(L^3)} \lesssim  \|u_0\|_{\dot{B}^{\frac{1}{2}}_{2, 1}}
\end{split}
\end{equation}
and \begin{equation}\label{col-pressure-Lip-1}
\|\nabla u\|_{L^1_T(L^\infty)} \lesssim\|u_0\|_{\dot{B}^{\frac{1}{2}}_{2, 1}}.
\end{equation}
}
\end{col}

\begin{proof}
 We first deduce from \eqref{est-basic-2-66} and \eqref{est-basic-2-104} that
\begin{equation*}\label{pressure-u-12}
\begin{split}
 \|t^{\frac{3}{4}}(\Delta\,u,\,\nabla\,\Pi)\|_{L^{4, 1}_T(L^3)}\lesssim  \|t^{\frac{1}{2}}(\Delta\,u,\,\nabla\,\Pi)\|_{L^{4, 1}_T(L^2)}^{\frac{1}{2}} \|t\,(\Delta\,u,\,\nabla\,\Pi)\|_{L^{4, 1}_T(L^6)}^{\frac{1}{2}}\lesssim  \|u_0\|_{\dot{B}^{\frac{1}{2}}_{2, 1}},
\end{split}
\end{equation*}
from which and  Proposition \ref{Neil}, we infer
\begin{equation}\label{pressure-u-13}
\begin{split}
\|(\Delta\,u,\,\nabla\,\Pi)\|_{L^1_T(L^3)} &\lesssim  \|t^{-\frac{3}{4}}\|_{L^{\frac{4}{3}, \infty}(\mathbb{R}^+)} \|t^{\frac{3}{4}}(\Delta\,u,\,\nabla\,\Pi)\|_{L^{4, 1}_T(L^3)}
\\&
\lesssim
\|t^{\frac{3}{4}}(\Delta\,u,\,\nabla\,\Pi)\|_{L^{4, 1}_T(L^3)}
\\&
\lesssim  \|u_0\|_{\dot{B}^{\frac{1}{2}}_{2, 1}}.
\end{split}
\end{equation}

To control $\|\nabla u\|_{L^1_T(L^{\infty})} $, we get, by using the interpolation inequality, that
\begin{equation}\label{est-Lip-1}
\begin{split}
\int_0^T\|\nabla u\|_{L^\infty} \,dt
&\lesssim
\int_0^T\|\nabla^2 u\|_{L^2}^{\frac{1}{2}} \|\nabla^2 u\|_{L^6}^{\frac{1}{2}}\,dt
\lesssim
\int_0^T t^{-\frac{3}{4}}\,\|t^{\frac{1}{2}}\nabla^2 u\|_{L^2}^{\frac{1}{2}} \|t\,\nabla^2 u\|_{L^6}^{\frac{1}{2}}\,dt
\\&
\lesssim
\| t^{-\frac{3}{4}}\|_{L^{\frac{4}{3}, \infty}(\mathbb{R}^+)}\,
\|t^{\frac{1}{2}}\nabla^2 u\|_{L^{4, 1}_T(L^2)}^{\frac{1}{2}} \|t\,\nabla^2 u\|_{L^{4, 1}_T(L^6)}^{\frac{1}{2}}
\\&
\lesssim\|u_0\|_{\dot{B}^{\frac{1}{2}}_{2, 1}},
\end{split}
\end{equation}
where we used \eqref{est-basic-2-66} and \eqref{est-basic-2-104}  in the last inequality.

Finally, we deduce  from the momentum equations in \eqref{1.2}, \eqref{est-basic-2-18}, \eqref{pressure-u-13}, and \eqref{est-Lip-1}, that
\begin{equation*}\label{pressure-u-13-aaa}
\begin{split}
\|\partial_t u\|_{L^1_T(L^3)}
&\lesssim
\|(\Delta\,u,\,\nabla\,\Pi)\|_{L^1_T(L^3)}+\| u\cdot\nabla\,u\|_{L^1_T(L^3)}
\\&
\lesssim
\|(\Delta\,u,\,\nabla\,\Pi)\|_{L^1_T(L^3)}+\| u\|_{L^{\infty}_T(\dot{B}^{\frac{1}{2}}_{2, 1})} \|\nabla\,u\|_{L^1_T(L^{\infty})}
\\&
\lesssim
\|u_0\|_{\dot{B}^{\frac{1}{2}}_{2, 1}}.
\end{split}
\end{equation*}
This together with \eqref{pressure-u-13} and \eqref{est-Lip-1}  finishes the proof of Corollary \ref{S4col1}.
\end{proof}

\begin{lem}\label{lem-density-est-1}
{\sl Under the assumptions of Corollary  \ref{S4col1}, one has
\begin{equation}\label{est-density-121-aaa}
\begin{split}
&\|a\|_{\widetilde L^\infty_T(\dot B^{\frac{3}{\lambda}}_{\lambda, 1})}
\leq
\|a_0\|_{\dot B^{\frac{3}{\lambda}}_{\lambda, 1}}
\exp\bigl(C\|u_0\|_{\dot{B}^{\frac{1}{2}}_{2, 1}}\bigr) \quad \mbox{if}\quad 3<\lambda <+\infty,\\
&\|a\|_{\widetilde L^\infty_T(\dot B^{0}_{\infty, 1})}
\leq
\|a_0\|_{\dot B^{0}_{\infty, 1}}\bigl(1+C\|u_0\|_{\dot{B}^{\frac{1}{2}}_{2, 1}}\bigr).
\end{split}
\end{equation}}
\end{lem}

\begin{proof}
It follows from Lemma 2.100 of \cite{BCD} that
  \begin{equation*}\label{est-density-11}
\begin{split}
&\|a\|_{\widetilde L^\infty_T(\dot B^{\frac{3}{\lambda}}_{\lambda, 1})}
\leq
\|a_0\|_{\dot B^{\frac{3}{\lambda}}_{\lambda, 1}}
\exp\bigl(C\|\nabla u\|_{L^1_T(L^\infty)}\bigr) \quad \mbox{if}\quad 3<\lambda <+\infty,\\
&\|a\|_{\widetilde L^\infty_t(\dot B^{0}_{\infty, 1})}
\leq
\|a_0\|_{\dot B^{0}_{\infty, 1}}\bigl(1+C\|\nabla u\|_{L^1_t(L^\infty)}\bigr).
\end{split}
\end{equation*}
By plugging \eqref{est-Lip-1} into the above inequalities, we obtain \eqref{est-density-121-aaa}.
\end{proof}

\begin{rmk}\label{rmk-pressure-1}
 By applying Lemmas \ref{lem-density-1} and \ref{lem-density-est-1} and Corollary \ref{S4col1}, we find
\begin{equation}\label{pressure-u-11-aaa}
\begin{split}
\sum_{j\in\mathbb{Z}}2^{j\left(-1+\frac{3}{p}\right)}\|[\dot\Delta_{j}\mathbb{P}, \rho^{-1}] (\Delta\,u-\nabla\,\Pi)\|_{L^1_t(L^p)}
&\leq C\|a\|_{\widetilde{L}^\infty_t(\dot{B}^{\frac{3}{\lambda}}_{\lambda, 1})}\|(\Delta\,u-\nabla\,\Pi)\|_{{L}^1_t(L^3)}\\
&\leq C \|a_0\|_{\dot B^{\frac{3}{\lambda}}_{\lambda, 1}}
\|u_0\|_{\dot{B}^{\frac{1}{2}}_{2, 1}}\exp\bigl(C\|u_0\|_{\dot{B}^{\frac{1}{2}}_{2, 1}}\bigr),
\end{split}
\end{equation}
when $2\le p<3$ and $3<\lambda<\infty$ such that
$\frac{1}{p}\le\frac{1}{3}+\frac{1}{\lambda}.$
\end{rmk}

\begin{prop}\label{prop-est-u-Besov-1}
{\sl  Under the assumptions of Corollary  \ref{S4col1},   we have
\begin{equation}\label{est-basic-2-244}
\begin{split}
&\|u\|_{\widetilde{L}^\infty_T(\dot{B}^{\frac{1}{2}}_{2, 1})}
+ \|u\|_{{L}^1_T(\dot{B}^{\frac{5}{2}}_{2, 1})}
\lesssim
\|u_0\|_{\dot{B}^{\frac{1}{2}}_{2, 1}}
\Bigl(1+\|a_0\|_{\dot B^{\frac{1}{2}}_{6, 1}}
\exp\bigl(C\|u_0\|_{\dot{B}^{\frac{1}{2}}_{2, 1}}\bigr)\Bigr).
\end{split}
\end{equation}}
\end{prop}

\begin{proof} We first get, by applying Leray projection operator $\mathbb{P}\eqdefa {\rm Id}+\na(-\D)^{-1}\dive$ to the equation
$\partial_t u + u \cdot \grad u+ \rho^{-1}(\grad \Pi-\Delta u)=0,$  that
\begin{equation*}
\begin{split}
&\partial_t u + \mathbb{P}(u \cdot \grad u)+ \mathbb{P}[\rho^{-1}(\grad \Pi-\Delta u)]=0.
\end{split}
\end{equation*}
Applying dyadic operator to the above equation yields
\begin{equation*}\label{est-basic-2-111}
\partial_t\dot{\Delta}_j u + \dot{\Delta}_j\mathbb{P}(u \cdot \grad u)+ \dot{\Delta}_j\mathbb{P}\bigl(\rho^{-1}(\grad \Pi-\Delta u)\bigr)=0.
\end{equation*}
By multiplying the above equation by $\rho$ and using a standard commutator's argument, we write
\begin{equation*}\label{est-basic-2-112}
\begin{split}
&\rho\partial_t\dot{\Delta}_j u+\rho  u \cdot \nabla\dot{\Delta}_ju - \Delta \dot{\Delta}_j u+ \rho[\dot{\Delta}_j\mathbb{P},\,u ]\cdot \grad u
+\rho [\dot{\Delta}_j\mathbb{P}, \rho^{-1}](\grad \Pi-\Delta u)=0.
\end{split}
\end{equation*}
Then we get, by taking $L^2$ inner product of the above equation with $\dot{\Delta}_j u,$ that
\begin{equation*}
\begin{split}
\frac{d}{dt}\|\dot{\Delta}_j u\|_{L^2}^2+&c_1\,2^{2j}\|\dot{\Delta}_j u\|_{L^2}^2\\
&\lesssim
\Bigl(\|[\dot{\Delta}_j\mathbb{P},\,u ]\cdot \grad u\|_{L^2}+ \|[\dot{\Delta}_j\mathbb{P}, \rho^{-1}](\grad \Pi-\Delta u)\|_{L^2}\Bigr)\|\dot{\Delta}_j u\|_{L^2},
\end{split}
\end{equation*}
from which, we infer
\begin{equation*}\label{est-basic-2-113}
\begin{split}
\|\dot{\Delta}_j u\|_{L^\infty_t(L^2)}&+ 2^{2j}\|\dot{\Delta}_j u\|_{L^1_t(L^2)}\\
&\lesssim \|\dot{\Delta}_j u_0\|_{L^2}+\|[\dot{\Delta}_j\mathbb{P},\,u ]\cdot \grad u\|_{L^1_t(L^2)}+ \|[\dot{\Delta}_j\mathbb{P}, \rho^{-1}](\grad \Pi-\Delta u)\|_{L^1_t(L^2)}.
\end{split}
\end{equation*}
So that we have
\begin{equation*}\label{est-basic-2-114}
\begin{split}
\|u\|_{\widetilde{L}^\infty_t(\dot{B}^{\frac{1}{2}}_{2, 1})}+ \|u\|_{\widetilde{L}^1_t(\dot{B}^{\frac{5}{2}}_{2, 1})}
\lesssim &
\|u_0\|_{\dot{B}^{\frac{1}{2}}_{2, 1}}+\sum_{j\in \mathbb{Z}}2^{\frac{j}{2}}\|[\dot{\Delta}_j\mathbb{P},\,u ]\cdot \grad u\|_{L^1_t(L^2)}\\
&+ \sum_{j\in \mathbb{Z}}2^{\frac{j}{2}}\|[\dot{\Delta}_j\mathbb{P}, \rho^{-1}](\grad \Pi-\Delta u)\|_{L^1_t(L^2)},
\end{split}
\end{equation*}
which together with \eqref{est-basic-2-121} and
\eqref{pressure-u-11-aaa} ensures that
\begin{equation*}\label{est-basic-2-205}
\begin{split}
\|u\|_{\widetilde{L}^\infty_t(\dot{B}^{\frac{1}{2}}_{2, 1})}
+ \|u\|_{{L}^1_t(\dot{B}^{\frac{5}{2}}_{2, 1})}
\lesssim &
\|u_0\|_{\dot{B}^{\frac{1}{2}}_{2, 1}}+ \|u\|_{{L}^{1}_t(\dot{B}^{1}_{\infty, 1})} \|u\|_{{L}^{\infty}_t(\dot{B}^{\frac{1}{2}}_{2, \infty})}\\
&+\|a_0\|_{\dot B^{\frac{1}{2}}_{6, 1}}
\exp\bigl(C\|u_0\|_{\dot{B}^{\frac{1}{2}}_{2, 1}}\bigr)\|u_0\|_{\dot{B}^{\frac{1}{2}}_{2, 1}}.
\end{split}
\end{equation*}
Due to \eqref{smallness-u-1}, $\|u_0\|_{\dot{B}^{\frac{1}{2}}_{2, \infty}}$ is sufficiently small, we deduce \eqref{est-basic-2-244}
from the above inequality.  This completes the proof of the proposition.
\end{proof}

Let us turn to the proof of Theorem \ref{thmmain-regularity}.

\begin{proof}[Proof of Theorem \ref{thmmain-regularity}]
Under the assumptions of Theorem \ref{thmmain-regularity}, we deduce from the classical theory of inhomogeneous
Navier-Stokes equations (see \cite{A-G-Z-2, DW2023} for instance), Lemma \ref{lem-density-est-1} and Proposition \ref{prop-est-u-Besov-1} that the system \eqref{1.2} has a unique local solution $(\rho, u)$ on $[0,T^\star)$ so that there holds
\eqref{bdd-density-res-1} and
\beno (\rho^{-1}-1, u) \in C([0, T^\star); \dot{B}^{\frac{1}{2}}_{6, 1}) \times \bigl(C([0, T^\star); \dot{B}^{\frac{1}{2}}_{2, 1})
\cap {L}^1([0,T^\star); \dot{B}^{\frac{5}{2}}_{2, 1})\bigr), \eeno
where $T^\star$ denotes the lifespan of the solution $(\rho,u).$

Furthermore due to \eqref{col-pressure-Lip-1}, we get, by a continuous argument, that $T^\star=\infty.$
The estimate \eqref{energy-ineq-total-33} follows from \eqref{est-density-121-aaa} for $\lambda=6$ and \eqref{est-basic-2-244}.
This completes the proof of Theorem \ref{thmmain-regularity}.
\end{proof}

\renewcommand{\theequation}{\thesection.\arabic{equation}}
\setcounter{equation}{0}
\appendix

\section{Tools box on Littlewood-Paley analysis and Lorentz spaces}\label{Sect5}

In this section, we recall some basic facts on Littlewood-Paley theory from \cite{BCD} and Lorentz spaces from \cite{G14}.
 Let us briefly explain how it may be built in the
case $x\in\R^3$.

Let $\chi(\tau)$ and $\varphi(\tau)$ be smooth functions such that
\begin{align*}
&\Supp \varphi \subset \Bigl\{\tau \in \R\,: \, \frac34 <
\tau < \frac83 \Bigr\}\quad\mbox{and}\quad \forall
 \tau>0\,,\ \sum_{q\in\Z}\varphi(2^{-q}\tau)=1,
\end{align*}
we define the dyadic operators as follows:
for $u\in{\mathcal S}_h'\eqdefa\bigl\{u\in {\mathcal S}', \ \lim\limits_{\lambda\rightarrow +\infty}\|\theta(\lambda\,D)u\|_{L^\infty}=0\,\,\text{for any}\,\,\theta \in \mathcal{D}(\mathbb{R}^3)\bigr\}$,
\begin{equation}\label{LP-decom-sum-1}
  \begin{aligned}
&\dot\Delta_qu\eqdefa\varphi(2^{-q}|\textnormal{D}|)u\ \ \forall q\in\Z,\hspace{1cm}\mbox{and}
\hspace{1cm}
\dot S_qu\eqdefa\sum_{j \leq q-1}\dot\Delta_{j}u,
\end{aligned}
\end{equation}
The dyadic operators satisfy the
property of almost orthogonality:
\begin{equation*}
\begin{split}
&\dot\Delta_k\dot\Delta_q u\equiv 0
\quad\mbox{if}\quad\vert k-q\vert\geq 2
\quad\mbox{and}\quad\dot\Delta_k(\dot S_{q-1}u\dot\Delta_q u)
\equiv 0\quad\mbox{if}\quad\vert k-q\vert\geq 5,
\end{split}
\end{equation*}
\begin{defi}\label{def1.1}
{\sl Let $s\in\R$, $1 \leq p,r\leq +\infty,$ we define
 the homogeneous Besov space $\dot
B^s_{p,r}$ to be the set of  distributions $u$ in ${\mathcal S}_{h}'$
 (${\mathcal S}_{h}'\eqdefa\{u\in {\mathcal S}', \ \lim\limits_{\lambda\rightarrow +\infty}\|\theta(\lambda\,D)u\|_{L^\infty}=0\,\,\text{for any}\,\,\theta \in \mathcal{D}(\mathbb{R}^3)\}$) so that
\begin{equation*}
\|u\|_{\dot B^s_{p,r}}\eqdefa\Big\|2^{qs}\|\dot\Delta_q
u\|_{L^{p}}\Big\|_{\ell ^{r}(\mathbb{Z})}<\infty.
\end{equation*}
 }
\end{defi}

\begin{rmk}\label{rmk1.1}
\begin{enumerate}
  \item We point out that if $s>0$ then the inhomogeneous Besov space $B^s_{p,r}=\dot B^s_{p,r}\cap L^p$ and
$$
\|u\|_{B^s_{p,r}}\approx \|u\|_{\dot B^s_{p,r}}+\|u\|_{L^p}.
$$

\item If $u \in \dot B^s_{ p,\infty} \cap \dot B^{\tilde{s}}_{
p,\infty}$ with $s < \tilde{s}$,  $1 \leq p\leq \infty,$ then for any $\theta \in (0, 1)$, $u
\in  \dot B^{\theta s+ (1-\theta)\tilde{s}}_{ p,1}$ and
\begin{equation}\label{interpo-complex-1}
\|u\|_{\dot B^{\theta
s+ (1-\theta)\tilde{s}}_{ p,1}} \leq \frac{C}{\tilde{s}-s}\bigl(\frac{1}{\theta}+\frac{1}{1-\theta}\bigr)\|u\|_{\dot B^s_{ p,\infty}}^{\theta}
\|u\|_{\dot B^{\tilde{s}}_{ p,\infty}}^{1-\theta}.
\end{equation}
  \item Let $s\in \mathbb{R}, 1\le p,\,r\leq+\infty$, and $u \in
\cS'_h.$ Then $u$ belongs to $\dot{B}^{s}_{p, r}$ if and
only if there exists some positive constant $C$ and some nonnegative generic element $\{c_{q, r}\}_{q \in \mathbb{Z}} $ of $\ell^r(\Z)$ such that
$\|\{c_{q, r}\}_{q\in\Z}\|_{\ell^{r}(\Z)} =1$ and for any $q\in \mathbb{Z}$
\begin{equation*}
\|\dot{\Delta}_{q}u\|_{L^{p}}\leq C c_{q, r} \, 2^{-q s }
\|u\|_{\dot{B}^{s}_{p, r}}.
\end{equation*}
\end{enumerate}
\end{rmk}

We also recall Bernstein's lemma from  \cite{BCD}:

\begin{lem}\label{lem2.1}
{\sl Let $\mathcal{B}\eqdefa \{
\xi\in\R^3,\ |\xi|\leq\frac{4}{3}\}$ be a ball   and $\mathcal{C}\eqdefa \{
\xi\in\R^3,\frac{3}{4}\leq|\xi|\leq\frac{8}{3}\}$ be an annulus.
 A constant $C$ exists so that for any positive real number $\lambda,$ any nonnegative
integer $k,$ any smooth homogeneous function $\sigma$ of degree $m$,
any couple of real numbers $(a, \; b)$ with $ b \geq a \geq 1$, and any function $u$ in $L^a$,
there hold
\begin{equation}
\begin{split}
&\Supp \hat{u} \subset \lambda \mathcal{B} \Rightarrow
\sup_{|\alpha|=k} \|\pa^{\alpha} u\|_{L^{b}} \leq  C^{k+1}
\lambda^{k+ 3\left(\frac{1}{a}-\frac{1}{b} \right)}\|u\|_{L^{a}},\\
& \Supp \hat{u} \subset \lambda \mathcal{C} \Rightarrow
C^{-1-k}\lambda^{ k}\|u\|_{L^{a}}\leq
\sup_{|\alpha|=k}\|\partial^{\alpha} u\|_{L^{a}}\leq
C^{1+k}\lambda^{ k}\|u\|_{L^{a}},\\
& \Supp \hat{u} \subset \lambda \mathcal{C} \Rightarrow \|\sigma(D)
u\|_{L^{b}}\leq C_{\sigma, m} \lambda^{ m+3\left(\frac{1}{a}-\frac{1}{b}
\right)}\|u\|_{L^{a}}, \end{split}\label{2.1}
\end{equation}}
with $\sigma(D)
u\eqdefa\mathcal{F}^{-1}(\sigma\,\hat{u})$.
\end{lem}

In what follows, we shall frequently use Bony's
decomposition \cite{Bony} in the  homogeneous context, which reads
\begin{equation}\label{bony}
\begin{split}
uv=& T_u v+T'_vu=T_u v+T_v u+R(u,v)\with\\
&T_u v\eqdefa\sum_{q \in \mathbb{Z}}\dot S_{q-1}u\dot\Delta_q v,\,
T'_vu\eqdefa\sum_{q \in \mathbb{Z}}\dot\Delta_q u\,\dot S_{q+2}v,\andf\\
 &R(u,v)\eqdefa\sum_{q \in \mathbb{Z}}\dot\Delta_q u {\widetilde{\dot\Delta}}_{q}v \,\,\text{with} \,\, {\widetilde{\dot\Delta}}_{q}v\eqdefa
\sum_{|q'-q|\leq 1}\dot\Delta_{q'}v.
\end{split}
\end{equation}

\begin{prop}[Theorems 2.47 and 2.52 in \cite{BCD}]\label{prop2.2}
{\sl \begin{enumerate}
\item  There exits  a constant $C$ so that for $s \in \mathbb{R}$, $t<0$, $p,\,p_1,\,p_2,\,r,\, r_1,\, r_2\in [1, +\infty]$,
\begin{equation*}
\begin{split}
&\|T_uv\|_{\dot{B}^s_{p, r}} \leq C^{|s|+1}\|u\|_{L^{\infty}}\|v\|_{\dot{B}^s_{p, r}},\\
 &\|T_uv\|_{\dot{B}^{s+t}_{p, r}} \leq \frac{C^{|s+t|+1}}{-t} \|u\|_{\dot{B}^t_{p_1, r_1}}\|v\|_{\dot{B}^s_{p_2, r_2}}\quad \mbox{with}
 \quad \frac{1}{p} \eqdefa \frac{1}{p_1}+\frac{1}{p_2},\quad \frac{1}{r} \eqdefa \min\Bigl(1, \frac{1}{r_1}+\frac{1}{r_2}\Bigr).
    \end{split}
\end{equation*}
\item
Let $(s_1, s_2)$ be in $\mathbb{R}^2$ and $(p_1, p_2, r_1, r_2)$ be in $[1,+\infty]^4$. We assume that
$\frac{1}{p} \eqdefa \frac{1}{p_1}+\frac{1}{p_2}\leq 1$ and $\frac{1}{r} \eqdefa \frac{1}{r_1}+\frac{1}{r_2}\leq 1$.
Then there exits  a constant $C$ so that
\begin{equation*}
\begin{split}
&\|R(u, v)\|_{\dot{B}^{s_1+s_2}_{p, r}} \leq  \frac{C^{s_1+s_2+1}}{s_1+s_2} \|u\|_{\dot{B}^{s_1}_{p_1, r_1}} \|v\|_{\dot{B}^{s_2}_{p_2, r_2}} \quad \text{if}\quad  s_1 + s_2 > 0,\\
&\|R(u, v)\|_{\dot{B}^{0}_{p, \infty}} \leq  C \|u\|_{\dot{B}^{s_1}_{p_1, r_1}} \|v\|_{\dot{B}^{s_2}_{p_2, r_2}} \quad \text{if}\quad r = 1 \,\,\text{and}\,\, s_1 + s_2 = 0.
    \end{split}
\end{equation*}
\end{enumerate}}
\end{prop}

In  order to obtain a better description of the regularizing effect
of the transport-diffusion equation, we shall use Chemin-Lerner type
norm from
\cite{CL}.
\begin{defi}\label{def2.2}
{\sl Let $s\in\R$,
$r,\lambda, p\in [1,+\infty]$ and $T>0$.
 we define
\begin{equation*}
\|u\|_{\widetilde L^\lambda_T(\dot B^s_{p,r})}\eqdefa\Big\|2^{qs}\|\dot\Delta_q
u\|_{L^\lambda_T(L^{p})}\Big\|_{\ell ^{r}(\mathbb{Z})}.
\end{equation*}
}
\end{defi}

Next we recall  the following commutator's estimate which will be frequently used throughout this paper.

\begin{lem}[ Lemma 2.97 in \cite{BCD}]\label{lem-commutator-1} {\sl Let $(p, q, r) \in [1, +\infty]^3$
satisfy $\frac{1}{r}=\frac{1}{p}+\frac{1}{q},$ $\theta$ be a $C^1$ function on $\mathbb{R}^{d}$ such that $(1+|\cdot|) \hat{\theta} \in L^1$. There
exists a constant $C$ such that for any function a with gradient in $L^p$
and any function $b$ in $L^q$, we have, for any positive $\lambda$,
\begin{equation}\label{commutator-compact-0}
\|[\theta(\lambda^{-1} D), a]b\|_{L^r} \leq C \lambda^{-1}\|\grad a\|_{L^p}\|b\|_{L^s}.
\end{equation}}
\end{lem}

Before recalling the definition of the Lorentz space, we begin by introducing the rearrangement  of a function (see \cite{G14} for instance). For a measurable function $f,$ we define its non-increasing rearrangement  by
 $f^{*}:\R_+\to \R_+$ by
$$
f^{*}(\lambda)\eqdefa\inf\Big\{s\geq0;\,
\big|\{\ x\in\R^3:\,|f(x)|>s\ \}\big|\leq\lambda\Big\},
$$
where $\big|\{\ x\in\R^3:\,|f(x)|>s\ \}\big|$ denotes the Lebesgue measure of the set $\{\ x\in\R^3:\,|f(x)|>s\ \}.$

\begin{defi} (Lorentz spaces)\label{espace_lorentz}
Let $f$ a mesurable function and
$1\leq p,q\leq\infty.$
Then $f$ belongs to the Lorentz space $L^{p,q}$ if
\begin{displaymath}
\|f\|_{L^{p,q}}\overset{def}{=}
\begin{cases}
            \Big( \int^\infty_0(t^{1\over p}f^*(t))^q{dt\over t}\Big)
            ^{1\over q}<\infty&
\text{if $q<\infty$}\\
             \displaystyle\sup_{t>0}\bigl(t^{1\over p}f^*(t)\bigr)<\infty &\text
             {if $q=\infty$}.
        \end{cases}
\end{displaymath}
\end{defi}
Alternatively, we can define the Lorentz spaces by the real interpolation, as the interpolation between the Lebesgue space~:
$$
L^{p,q}\eqdefa (L^{p_0},L^{p_1})_{(\theta,q)},
$$
with $1\le p_0<p<p_1\le\infty,$ $0<\theta<1$ satisfying
${1\over p}={1-\theta\over p_0}+{\theta\over p_1}$ and $1\leq q\leq\infty,$ also $f\in L^{p,q}$ if the following quantity
$$
\|f\|_{L^{p,q}}\eqdefa
\Big(\int_0^\infty\big(t^{-\theta}K(t,f)\big)^q
{dt\over t}\Big)^{1\over q}
$$
is finite with
$$
K(f,t)\eqdefa\displaystyle\inf_{f=f_0+f_1}
\big\{\|f_0\|_{L^{p_0}}+t\|f_1\|_{L^{p_1}}\;\,\big|
\;f_0\in L^{p_0},\,f_1\in L^{p_1}\big\}.
$$

The Lorentz spaces verify the following properties
(see  \cite{lema,ON} for more details)~:

\begin{prop}\label{Neil}
{\sl Let $f\in L^{p_1,q_1},$ $g\in L^{p_2,q_2}$ and
$1\leq p,q,p_j,q_j\leq\infty,$ for $1\leq j\leq2.$
\vspace{0,5cm}

\begin{itemize}
\item[(1)] If $1<p<\infty$ and $1\le q\le\infty,$ then
$$
\|fg\|_{L^{p,q}}
\lesssim
\|f\|_{L^{p,q}}\|g\|_{L^{\infty}}.
$$

\item[(2)]
If ${1\over p}={1\over p_1}+{1\over p_2}$ and
${1\over q}={1\over q_1}+{1\over q_2},$ then
$$
\|fg\|_{L^{p,q}}
\lesssim
\|f\|_{L^{p_1,q_1}}\|g\|_{L^{p_2,q_2}}.
$$

\item[(3)]
For $1\leq p\leq\infty$ and $1\leq q_1\leq q_2\leq\infty,$ we have
$$
L^{p,q_1}\hookrightarrow L^{p,q_2}
\hspace{1cm}\mbox{and}\hspace{1cm}L^{p,p}=L^p.
$$
\end{itemize}}
\end{prop}

\begin{prop}\label{lorentz}
{\sl  The following imbedding relations hold
 \begin{itemize}
\item[(1)] $\dot H^1(\R^3)\hookrightarrow  L^{6,2}(\R^3).$

\item[(2)] $\dot B^{{3\over p}-1}_{p,r}(\R^3)\hookrightarrow  L^{3,r}(\R^3),$
 for $1\leq p<3$ and $r\in[1,\infty].$
\end{itemize}}
\end{prop}

\begin{proof} The first part can be found in \cite{Tartar-1998}.  We focus on the proof of the second part.

 Let $(p, p_1)$ be a fixed   exponents pair   satisfying $
1\leq p<3<p_1\leq\infty$ and $r\in[1,\infty]$.
By definition, we have
 \beq\label{App1}
(L^{p},L^{p_1})_{(\theta,r)}=L^{3,r},\quad\hbox{
with }\quad{1\over3}=\frac{1-\theta}{ p}
+{\theta\over p_1}\cdot
\eeq
While it follows from Lemma \ref{lem2.1} that
$$
\dot B^{{3\over p}-{3\over p_1}}_{p,1}
\hookrightarrow\dot B^0_{p_1,1}\hookrightarrow L^{p_1}
\quad\mbox{and}\quad
\dot B^0_{p,1}\hookrightarrow L^{p},
$$
from which and \eqref{App1}, we infer
\beq\label{App2}
\bigl(\dot B^0_{p,1},B^{{3\over p}-{3\over p_1}}_{p,1}\bigr)_{(\theta,r)}
\hookrightarrow L^{3,r}.
\eeq

On the other hand, we have    (see for instance \cite{BE} page 152),
$$
\bigl(\dot B^0_{p,1},\dot B^{{3\over p}-{3\over p_1}}_{p,1}\bigr)_{(\theta,r)}
=\dot B^{\theta({3\over p}-{3\over p_1})}_{p,r}
=\dot B^{{3\over p}-1}_{p,r},
$$
which together with \eqref{App2} ensures the second part of the lemma.
\end{proof}


\noindent {\bf Acknowledgments.}
  G. Gui is supported in part by National Natural Science Foundation of China under Grants 12371211 and 12126359. P. Zhang is supported by National Key R$\&$D Program of China under grant 2021YFA1000800 and National Natural Science Foundation of China under Grants  12288201 and 12031006.

\end{document}